\documentclass[11pt]{amsart}

\usepackage{enumerate}

\usepackage[foot]{amsaddr}
\usepackage[english]{babel}
\newcommand{\lvt}{\left|\kern-1.35pt\left|\kern-1.3pt\left|}
\newcommand{\rvt}{\right|\kern-1.3pt\right|\kern-1.35pt\right|}

\usepackage{tensor}
\usepackage[foot]{amsaddr}
\usepackage[english]{babel}
\usepackage[%backref,
pdftex,hyperfootnotes]{hyperref}
\usepackage[usenames,dvipsnames,svgnames,table,x11names,table]{xcolor}
\usepackage{pgfplots}
\usepackage{tikz}
\usepackage{tikz-3dplot}
\usetikzlibrary{calc,shadows,shapes.callouts,shapes.geometric,shapes.misc,positioning,patterns,decorations.pathmorphing,decorations.markings,decorations.fractals,decorations.pathreplacing,shadings,fadings}

\usepackage{longtable}

\allowdisplaybreaks

\usepackage{fancybox}
\usepackage[utf8]{inputenc}

\usepackage{tcolorbox}

\usepackage{rotating,multirow,pdflscape}
\usepackage{amssymb,latexsym,amsmath,amsthm}
\usepackage{mathrsfs}
\usepackage{mathtools,arydshln,mathdots}
\usepackage{bigints}
\usepackage{comment}
\mathtoolsset{showonlyrefs}

\usepackage{nicematrix}

\makeatletter
\renewcommand*\env@matrix[1][*\c@MaxMatrixCols c]{%
 \hskip -\arraycolsep
 \let\@ifnextchar\new@ifnextchar
 \array{#1}}
\makeatother

\mathtoolsset{centercolon}
\usepackage[toc,page]{appendix}

\usepackage{tocvsec2}

\usepackage{Baskervaldx}
\usepackage{newtxmath}

\theoremstyle{plain}

\newtheorem{teo}{Theorem}[section]
\newtheorem{coro}[teo]{Corollary}
\newtheorem{lemma}[teo]{Lemma}
\newtheorem{pro}[teo]{Proposition}

\theoremstyle{defi}

\theoremstyle{remark}
\newtheorem{rem}[teo]{Remark}

\newcommand{\ii}{\operatorname{i}}

\renewcommand{\d}{\operatorname{d}}

\renewcommand{\Re}{\operatorname{Re}}

\newcommand{\C}{\mathbb{C}}

\newcommand{\R}{\mathbb{R}}

\newcommand{\sz}[1]{\left| \vec{#1} \right|}
\allowdisplaybreaks[1]

\makeatletter
\DeclareRobustCommand{\gaussk}{\DOTSB\gaussk@\slimits@}
\newcommand{\gaussk@}{\mathop{\vphantom{\sum}\mathpalette\bigcal@{K}}}

\newcommand{\bigcal@}[2]{%
 \vcenter{\m@th
 \sbox\z@{\( #1\sum\)}%
 \dimen@=\dimexpr\ht\z@+\dp\z@
 \hbox{\resizebox{!}{0.8\dimen@}{\( \mathcal{K}\)}}%
 }%
}
\newcommand{\cfracplus}{\mathbin{\cfracplus@}}
\newcommand{\cfracplus@}{%
 \sbox\z@{\( \dfrac{1}{1} \)}%
 \sbox\tw@{\( + \)}%
 \raisebox{\dimexpr\dp\tw@-\dp\z@\relax}{\(+\)}%
}
\newcommand{\cfracdots}{\mathord{\cfracdots@}}
\newcommand{\cfracdots@}{%
 \sbox\z@{\(\dfrac{1}{1}\)}%
 \sbox\tw@{\(+\)}%
 \raisebox{\dimexpr\dp\tw@-\dp\z@\relax}{\( \cdots\)}%
}
\makeatother

\makeatletter
\newcommand*{\relrelbarsep}{.386ex}
\newcommand*{\relrelbar}{%
 \mathrel{%
 \mathpalette\@relrelbar\relrelbarsep
 }%
}
\newcommand*{\@relrelbar}[2]{%
 \raise#2\hbox to 0pt{\( \m@th#1\relbar\( \hss}%
 \lower#2\hbox{\( \m@th#1\relbar\)}%
}
\providecommand*{\rightrightarrowsfill@}{%
 \arrowfill@\relrelbar\relrelbar\rightrightarrows
}
\providecommand*{\leftleftarrowsfill@}{%
 \arrowfill@\leftleftarrows\relrelbar\relrelbar
}
\providecommand*{\xrightrightarrows}[2][]{%
 \ext@arrow 0359\rightrightarrowsfill@{#1}{#2}%
}
\providecommand*{\xleftleftarrows}[2][]{%
 \ext@arrow 3095\leftleftarrowsfill@{#1}{#2}%
}
\makeatother

\newcommand*\pFqskip{8mu}
\catcode`,\active
\newcommand*\pFq{\begingroup
 \catcode`\,\active
 \def ,{\mskip\pFqskip\relax}%
 \dopFq
}
\catcode`\,12
\def\dopFq#1#2#3#4#5{%
 {}_{#1}F_{#2}\biggl[\genfrac..{0pt}{}{#3}{#4};#5\biggr]%
 \endgroup
}

\newcommand{\KF}[5]{F^{#1}_{#2}\left[{#3\atop #4}\Bigg\vert #5\right]}

\usepackage[textwidth=17.5cm,textheight=22.275cm,
height=%24cm,
24.275cm,
width=18cm]{geometry}

\usepackage{xcolor} %Cambia el color de la fuente

\usepackage[normalem]{ulem}
\usepackage{soul} %Fuente tachada

\usepackage{comment} %Para comentar cachos gordos

\usepackage{tikz}
\usetikzlibrary{shapes,arrows}
\usepackage{verbatim}

\tikzstyle{block} = [draw, rectangle, 
 minimum height=3em, minimum width=2em]
\usepackage{mathrsfs} %Para más letras raras

\title[Integral and hypergeometric representations of multiple orthogonality]{Integral and hypergeometric representations \\
for multiple orthogonal polynomials 
}



\author[Branquinho]{Amílcar Branquinho\( ^1\)}
\address{\( ^1\)CMUC, Departamento de Matem\'atica,
Universidade de Coimbra, Largo D. Dinis, 3000-143 Coimbra, Portugal}
\email{\(^1\)ajplb@mat.uc.pt}

\author[Díaz]{Juan EF Díaz\(^{2}\)}
\address{\(^2\)CIDMA, Departamento de Matemática, Universidade de Aveiro, 3810-193 Aveiro, Portugal}
\email{\(^2\)juan.enri@ua.pt}

\author[Foulquié]{Ana Foulquié-Moreno\(^3\)}
\address{\(^3\)CIDMA, Departamento de Matemática, Universidade de Aveiro, 3810-193 Aveiro, Portugal}
\email{\(^3\)foulquie@ua.pt}

\author[Mañas]{Manuel Mañas\(^4\)}
\address{\(^4\)Departamento de Física Teórica, Universidad Complutense de Madrid, 28040-Madrid, Spain}
\email{\(^4\)manuel.manas@ucm.es}

\author[Wolfs]{Thomas Wolfs\(^5\)}
\address{\(^5\)Department of Mathematics, KU Leuven, 3001 Leuven, Belgium}
\email{\(^5\)thomas.wolfs@kuleuven.be}

\begin{document}

\maketitle

\begin{abstract}
This paper addresses two primary objectives in the realm of classical multiple orthogonal polynomials with an arbitrary number of weights. Firstly, it establishes new and explicit hypergeometric expressions for type I Hahn multiple orthogonal polynomials.
Secondly, 
%utilizing 
applying the residue theorem and the Mellin transform, the paper derives contour integral representations for several families of orthogonal polynomials. Specifically, it presents contour integral formulas for both type I and type II multiple orthogonal polynomials in the Laguerre of the first kind, Jacobi--Piñeiro, and Hahn families. The evaluation of these integrals leads to explicit hypergeometric representations.
\end{abstract}

\tableofcontents

\section{Introduction}

The first four coauthors embarked on a program to find explicit hypergeometric expressions for both the continuous and discrete versions of classical multiple orthogonal polynomials. Initially, in \cite{HahnI}, they identified all families under the Hahn polynomials in the multiple Askey scheme, except for the Hermite family, for the more simple case of two weights. This included the Jacobi--Piñeiro, Laguerre of the first and second kinds, Meixner of the first and second kinds, Kravchuk, and Charlier polynomials. The focus then shifted to an arbitrary number of weights. Formulas for Jacobi--Piñeiro and Laguerre of the first kind were presented in~\cite{HS:JP-L1}, while those for Laguerre of the second kind and Hermite multiple orthogonal polynomials of type~I were given in \cite{CMOPI}. Notably, explicit expressions for type I Hermite multiple orthogonal polynomials were previously unknown even for two weights. Additionally, the recursion coefficients for all these cases were determined explicitly.

However, finding the Hahn multiple orthogonal polynomials of type I for an arbitrary number of weights proved elusive for the initial group of coauthors. A recent shift in perspective led to the breakthrough presented in this paper: for the first time, explicit hypergeometric expressions for the type I Hahn polynomials are provided.

Simultaneously, the last coauthor, as part of his PhD thesis supervised by Walter Van Assche, developed a contour integral technique based on the residue theorem and the Mellin transform to derive families of multiple orthogonal polynomials of types I and II \cite{VAWolfs,Wolfs}. Using this technique, he independently found contour integral representations for the Hahn multiple orthogonal polynomials that led to the same hypergeometric expressions for the type I Hahn polynomials. It also enabled him to obtain hypergeometric expressions for the type II Hahn polynomials different than those in \cite{AskeyII}.

At the Arno Kuijlaars 60th anniversary meeting in Leuven, discussions among several coauthors revealed that  had independently arrived at the same novel result using different perspectives and methods. Encouraged by Professor Van Assche, we decided to combine these developments and present both techniques side by side.

The paper is structured as follows.
We conclude this introduction by presenting some preliminary facts essential for the development of this study. Specifically, we discuss fundamentals of multiple orthogonal polynomials, hypergeometric-type series, the Mellin transform, and certain aspects related to the Askey scheme in the context of multiple orthogonal polynomials. 
%In Section \ref{sec:nova} we present some preliminary facts essential for the development of this study. Specifically, we discuss fundamentals of multiple orthogonal polynomials, hypergeometric-type series, the Mellin transform, and certain aspects related to the Askey scheme in the context of multiple orthogonal polynomials.
Section \ref{S:Hahn} focuses on the motivation and proof of one of the main results of this paper, Theorem \ref{HahnTypeITheorem}. This result provides for the first time explicit hypergeometric expressions for the type I Hahn multiple orthogonal polynomials.
In Section \ref{S: Integral}, we 
%utilize 
apply
the residue theorem and the Mellin transform to derive, in Theorems~\ref{HI_IR} and~\ref{HII_IR}, contour integral expressions for the Laguerre of the first kind, Jacobi--Piñeiro, and Hahn multiple orthogonal polynomials of both types~I and~II. These integral representations constitute another key result of the paper and lead to the hypergeometric explicit representations discussed earlier.

\subsection{Multiple Orthogonal Polynomials}
%\section{Multiple Orthogonal Polynomials} \label{sec:nova}
For the upcoming discussion, we require some basics about discrete multiple orthogonal polynomials \cite[Chapter~23.6]{Ismail}, see also \cite{nikishin_sorokin}. We begin by examining a system characterized by an arbitrary number, say \( p \), of weight functions, denoted as
\begin{align*}
w_1, \ldots, w_p: \Delta \subseteq \mathbb{R} \rightarrow \mathbb{R}_{\geq 0} .
\end{align*}In certain cases, a sequence of type II polynomials \( B_{\vec{n}} \) may exist, which are monic with \(\deg{B_{\vec{n}}} \leqslant |\vec{n}|\), satisfying the orthogonality relations 
\begin{align}
 \label{ortogonalidadTipoIIContinua}
\begin{aligned}
	 \int_{\Delta} x^j B_{\vec{n}}(x) w_i(x) \, \d x &= 0, & j &\in \{0, \ldots, n_i - 1\}, & i &\in \{1, \ldots, p\}.
\end{aligned}
\end{align}
Additionally, there are \( p \) sequences of type I polynomials, denoted as \( A^{(1)}_{\vec{n}}, \ldots, A^{(p)}_{\vec{n}} \), with \(\deg A^{(i)}_{\vec{n}} \leqslant n_i - 1\), meeting the orthogonality conditions
\begin{align}
 \label{ortogonalidadTipoIContinua}
 \sum_{i = 1}^p \int_{\Delta} x^j A^{(i)}_{\vec{n}}(x) w_i(x) \, \d x =
 \begin{cases}
 0, & \text{if } j \in \{0, \ldots, |\vec{n}| - 2\}, \\
 1, & \text{if } j = |\vec{n}| - 1.
 \end{cases}
\end{align}
In this context, \(\vec{n} = (n_1, \ldots, n_p) \in \mathbb{N}^p_0\) and \(|\vec{n}| \coloneqq n_1 + \cdots + n_p\). For this paper, we will use the notation \(\mathbb{N} \coloneqq \{1, 2, 3, \ldots\}\) and \(\mathbb{N}_0 \coloneqq \{0\} \cup \mathbb{N}\).

These conditions imply that the polynomials are uniquely defined and satisfy the biorthogonality conditions:
\begin{align*}
%\label{continuousbiorthogonality}
 \sum_{i = 1}^p \int_{ \Delta} B_{\vec{n}}(x) A^{(i)}_{\vec{m}}(x) w_i(x) \, \d x =
 \begin{cases}
 0, & \text{if } m_i \leqslant n_i, \; i \in \{1, \ldots, p\}, \\
 1, & \text{if } |\vec{m}| = |\vec{n}| + 1, \\
 0, & \text{if } |\vec{n}| + 1 < |\vec{m}|.
 \end{cases}
\end{align*}
Analogously, we can consider a system of discrete weight functions \(w_1,\ldots,w_p:\Delta\subseteq\mathbb Z\rightarrow\mathbb R_\geq\), leading to the discrete orthogonality relations
\begin{align}
 %\label{ortogonalidadTipoIIDiscreta}
 \nonumber
 \sum_{k\in\Delta} k^j B_{\vec{n}}(k) w_i(k) &= 0, \quad j \in \{0, \ldots, n_i - 1\}, \quad i \in \{1, \ldots, p\},
%\end{align}
%\begin{align} 
 \\
 \label{ortogonalidadTipoIDiscreta}
 \sum_{i = 1}^p \sum_{k\in\Delta} k^j A^{(i)}_{\vec{n}}(k) w_i(k) & =
 \begin{cases}
 0, & \text{if } j \in \{0, \ldots, |\vec{n}| - 2\}, \\
 1, & \text{if } j = |\vec{n}| - 1,
 \end{cases}
 \\
%\end{align}
%\begin{align}
%\label{discretebiorthogonality}
\nonumber
 \sum_{i = 1}^p \sum_{k\in \Delta} B_{\vec{n}}(k) A^{(i)}_{\vec{m}}(k) w_i(k) & =
 \begin{cases}
 0, & \text{if } m_i \leqslant n_i, \; i \in \{1, \ldots, p\}, \\
 1, & \text{if } |\vec{m}| = |\vec{n}| + 1, \\
 0, & \text{if } |\vec{n}| + 1 < |\vec{m}|.
 \end{cases}
\end{align}

\subsection{Hypergeometric Functions}

It's important to recall that, in many cases, these polynomials can be represented using generalized hypergeometric series \cite{andrews}:
\begin{align*}
%\label{GeneralizedHypergeometricFuntions}
 \pFq{p}{q}{a_1, \ldots, a_p}{\alpha_1, \ldots, \alpha_q}{x} \coloneqq \sum_{l=0}^{\infty} \dfrac{(a_1)_l \cdots (a_p)_l}{(\alpha_1)_l \cdots (\alpha_q)_l} \dfrac{x^l}{l!}.
\end{align*}
In other cases, it is necessary to go a step further and use the Kampé de Fériet series \cite{andrews, slater, Srivastava}:
\begin{multline}
 \label{KF}
 \KF{p:r;s}{q:n;k}{(a_1, \ldots, a_p):(b_1, \ldots, b_r);(c_1, \ldots, c_s)}{(\alpha_1, \ldots, \alpha_q):(\beta_1, \ldots, \beta_n);(\gamma_1, \ldots, \gamma_k)}{x, y}
\\
 \coloneqq \sum_{l=0}^{\infty} \sum_{m=0}^{\infty} \dfrac{(a_1)_{l+m} \cdots (a_p)_{l+m}}{(\alpha_1)_{l+m} \cdots (\alpha_q)_{l+m}} \dfrac{(b_1)_l \cdots (b_r)_l}{(\beta_1)_l \cdots (\beta_n)_l} \dfrac{(c_1)_m \cdots (c_s)_m}{(\gamma_1)_m \cdots (\gamma_k)_m} \dfrac{x^l}{l!} \dfrac{y^m}{m!}.
\end{multline}
Here, \((x)_n\), \(x \in \mathbb{C}\) and \(n \in \mathbb{N}_0\), known as the Pochhammer symbol, is defined as:
\begin{align*}
 (x)_n = \dfrac{\Gamma(x+n)}{\Gamma(x)} = \begin{cases}
 x (x+1) \cdots (x+n-1), & \text{if } n \in \mathbb{N}, \\
 1, & \text{if } n = 0.
 \end{cases}
\end{align*}
Recall that for \(\operatorname{Re} (z)>0\), the Eulerian gamma function has the integral representation:
\begin{align*}
	\Gamma (z)=\int _{0}^{\infty }t^{z-1}\operatorname e^{-t} \, \d t.
\end{align*}
The gamma function is subsequently defined as the analytic continuation of this integral to a meromorphic function that is analytic across the entire complex plane, except at zero and the negative integers, where it exhibits simple poles. The asymptotic Stirling formula for the  gamma function 
\begin{equation}\label{eq:Stirling}
\Gamma (z)=\sqrt {\frac {2\pi }{z}}
\left(\frac {z}{\operatorname e}\right)^z
\left(1+\operatorname{O}\left({\frac {1}{z}}\right)\right).
\end{equation}
is valid for  $|z|\to \infty $ in the sector  $|\arg(z)| < \pi - \epsilon$, $\epsilon>0$.

The Chu--Vandermode formula
\begin{align} \label{Chu-Van}
	(a+b)_n = \sum_{k=0}^n \binom{n}{k} (a)_k (b)_{n-k} ,
\end{align}
will be relevant later.

To establish the results related to the type I polynomials, we will rely on several results involving these functions. In order to deduce the guess for the main result, we will need the Kummer transformation for the \({}_3F_2\) function:
\begin{align}
 \label{Kummer}
 \pFq{3}{2}{a_1, a_2, a_3}{b_1, b_2}{1} 
 = \dfrac{\Gamma(b_2) \Gamma(b_1 + b_2 - a_1 - a_2 - a_3)}{\Gamma(b_2 - a_1) \Gamma(b_1 + b_2 - a_2 - a_3)} \, \pFq{3}{2}{a_1, b_1 - a_2, b_1 - a_3}{b_1, b_1 + b_2 - a_2 - a_3}{1}, 
\end{align}
with \( \Re(b_2), \Re(b_1 + b_2 - a_1 - a_2 - a_3) > 0 \),
as well as the following reduction formula due to Rakha and Rathie, cf.~\cite[equation 3.20]{Rakha-Rathie}:
\begin{align}
 \label{RR}
 \KF{2:1;2}{2:0;1}{\alpha, \lambda: \epsilon; \beta - \epsilon, \gamma}{\beta, \mu: --; \delta}{1, 1} = \dfrac{\Gamma(\mu) \Gamma(\mu - \alpha - \lambda)}{\Gamma(\mu - \alpha) \Gamma(\mu - \lambda)} \, \pFq{4}{3}{\alpha, \lambda, \beta - \epsilon, \delta - \gamma}{\beta, \delta, 1 - \mu + \alpha + \lambda}{1},
\end{align}
which was already used to prove the main result in \cite{HahnI}.

Finally, the proof of the main result of this paper, which generalizes the one in \cite{HahnI}, will rely on the following reformulated version (cf. \cite[Lemma 1]{HS:JP-L1}) of a theorem by Karp and Prilepkina \cite[Theorem 2.2]{KP}.

\begin{lemma}\label{lemma:KP}
 Let \(a \in \mathbb{C}\), \(\vec{f} = (f_1, \ldots, f_r) \in \mathbb{C}^r\), \(\vec{b} = (b_1, \ldots, b_l) \in \mathbb{C}^l\) with all components distinct, \(\vec{m} = (m_1, \ldots, m_r) \in \mathbb{N}^r\), and \(\vec{k} = (k_1, \ldots, k_l) \in \mathbb{N}^l\). If \(\Re(k_1 + \cdots + k_l - a - m_1 - \cdots - m_r) > 0\), then
 \begin{multline}
 \label{KPTheoremReformulated}
 \pFq{r+l+1}{r+l}{a, f_1 + m_1, \ldots, f_r + m_r, b_1, \ldots, b_l}{f_1, \ldots, f_r, b_1 + k_1, \ldots, b_l + k_l}{1} \\
 = \Gamma(1 - a) \dfrac{(b_1)_{k_1} \cdots (b_l)_{k_l}}{(f_1)_{m_1} \cdots (f_r)_{m_r}} \sum_{q=1}^{l} \dfrac{(-1)^{k_q-1} \prod_{j=1}^{r} (\tilde{f}_j - m_j)_{m_j} \Gamma(b_q + k_q - 1)}{(k_q - 1)! \prod_{j=1, j \neq q}^{l} (b_j - b_q - k_q + 1)_{k_j} \Gamma(b_q + k_q - a)} \\
 \times \pFq{r+l+1}{r+l}{-k_q + 1, -b_q - k_q + 1 + a, \tilde{f}_1, \ldots, \tilde{f}_r, \vec{b}^{\ast q} - (b_q + k_q - 1) \vec{e}_{l-1}}{-b_q - k_q + 2, \tilde{f}_1 - m_1, \ldots, \tilde{f}_r - m_r, \vec{b}^{\ast q} + \vec{k}^{\ast q} - (b_q + k_q - 1) \vec{e}_{l-1}}{1},
 \end{multline}
 with \(\tilde{f}_j \coloneq f_j - b_q + 1 - k_q + m_j\), \(j = 1, \ldots, r\).
\end{lemma}

\begin{rem}
 Notice that here it has been introduced the following notation. With the unit vector
  %is
represented~as
 \begin{align*}
 \vec{e}_p\coloneq (1,\ldots,1)\in\mathbb R^p ,
 \end{align*}
 %Then, 
given a vector \(\vec{a}\in\mathbb R^p\), then
\(\vec{a}^{\ast q}\in\mathbb R^{p-1}\) means that the \(q\)-th component has been removed. This notation will often appear along the rest of the text.
\end{rem}

\subsection{Mellin Transform} 
As observed in \cite{VAWolfs,Wolfs}, the Mellin transform serves as a useful tool in the study of multiple orthogonal polynomials. The Mellin transform of a function \( f \in L^1(\mathbb{R}_{\geqslant 0}) \) is given by
\begin{align*}
(\mathcal{M} f)(s) = \int_0^\infty f(x) x^{s-1} \, \d x.
\end{align*}
Typically, it is analytic in a strip \(\{s\in\C \mid a<\operatorname{Re}(s)<b\}\). In that case, for any \(c\in (a,b) \), there exists an inverse transform, namely
\begin{align*}
	f(x) = \int_{c- \ii \infty}^{c+ \ii \infty} (\mathcal{M} f)(s) x^{-s} \frac{\operatorname d s}{2 \pi \ii },\quad x\in\R_{\geq 0}.
\end{align*}
Some standard examples are the Mellin transform of a gamma density on \((0,\infty)\),
\begin{align}\label{eq:Mellin_exp}
	(\mathcal {M}\operatorname{e}^{-x})(s) = \Gamma(s),
\end{align}
and the Mellin transform of a beta density on \((0,1)\),
\begin{equation}\label{eq:Mellin_beta}
	(\mathcal{M}(1-x)^\beta \chi_{(0,1)}(x))(s) = \frac{\Gamma(s)\Gamma(\beta+1)}{\Gamma(s+\beta+1)}.
\end{equation}

In the discrete setting, where the underlying measure \(\nu_N=\sum_{k=0}^N \delta_k\) has mass points on a lattice
\( \mathcal{L}_N = \{0,\ldots,N\}\), we may view the transform
\begin{align*}
 (\mathcal{M}_N f)(s) = \int_0^\infty f(x) \frac{\Gamma(x+s)}{\Gamma(x+1)} \operatorname d \nu_N(x) ,
\end{align*}
as the analogue of the Mellin transform. Its inverse is described in the result below. 

\begin{teo} \label{DMT_INV}
 Let \(\mathcal{C}\subset\C\) be an anticlockwise contour enclosing once \([-N,0]\). Then, 
\begin{align*}
f(x) = \int_{\mathcal{C}} (\mathcal{M}_N f)(s) \frac{\Gamma(x+1)}{\Gamma(x+s+1)} \frac{ \operatorname d s }{2 \pi \ii },\quad x\in\mathcal{L}_N.
\end{align*}
\end{teo}
\begin{proof}
 By definition, we have
\begin{align*}
 (\mathcal{M}_N f)(s) = \sum_{k=0}^N \frac{f(k)}{k!} \Gamma(s+k).
 \end{align*}
Observe that the ratio \[\frac{\Gamma(s+k) }{\Gamma(x+s+1)}=\frac{(s)_k }{(s)_{x+1}}\] has no poles if \(k\geq x+1\) and has simple poles at \(s\in\{k,\ldots,x\}\) if \(k\leq x\). The residue theorem therefore implies that
\begin{align*}
 \int_{\mathcal{C}} (\mathcal{M}_N f)(s) \frac{\Gamma(x+1)}{\Gamma(x+s+1)} \frac{ \operatorname d s }{2 \pi \ii } = \sum_{k=0}^x \sum_{l=k}^x \frac{f(k)}{k!} \frac{(-1)^{l-k}}{(l-k)!} \frac{x!}{(x-l)!}. 
 \end{align*}
It then remains to note that
\begin{align*}
 \sum_{l=k}^x \frac{(-1)^{l-k}}{(l-k)!} \frac{x!}{(x-l)!} = \sum_{l=0}^{x-k} \frac{(-1)^l}{l!} \frac{x!}{(x-k-l)!} = \frac{x!}{(x-k)!} (1-1)^{x-k} = k! \delta_{k,x},
 \end{align*}
in order to obtain the desired result.
\end{proof}

%%%%%%%%%%%%%
\subsection{Askey Scheme}

Here is a summary of the previously known results that we aim to expand upon in this paper.

\subsubsection{Laguerre of the First Kind Family} \label{S:Laguerre of the first kind}
The Laguerre weights are given by:
\begin{align*}
	%\label{WeightsLaguerreI}
\begin{aligned}	w_{i}(x;\alpha_i) &= \operatorname{e}^{-x}x^{\alpha_i}, & 
		i &\in \{1, \ldots, p\}, & \Delta &= [0, \infty).
\end{aligned}
\end{align*}
Here, \(\vec{\alpha} = (\alpha_1, \ldots, \alpha_p)\) with \(\alpha_1, \ldots, \alpha_p > -1\). To establish an AT system, we require that \(\alpha_i - \alpha_j \notin \mathbb{Z}\) for \(i \neq j\).

The monic type II polynomials admit two equivalent expressions, as shown in \cite[\S 4]{AskeyII}:
%\begin{subequations}\label{eqs:Laguerre}
\begin{align}
	%\label{LaguerreFKTypeII}
	\nonumber
	& \hspace{.4cm}
	L_{\vec{n}}(x) = (-1)^{|\vec{n}|}\prod_{q=1}^p (\alpha_q+1)_{n_q} \sum_{l_1=0}^{n_1} \cdots \sum_{l_p=0}^{n_p}
	\prod_{q=1}^p \dfrac{(-n_q)_{l_q}}{l_q!}
	\dfrac{(\alpha_1+n_1+1)_{l_2+\cdots+l_p} \cdots (\alpha_{p-1}+n_{p-1}+1)_{l_p}}{(\alpha_1+1)_{l_1+\cdots+l_p} \cdots (\alpha_p+1)_{l_p}} x^{l_1+\cdots+l_p},\\
	\label{LaguerreFKTypeIIWeighted}
	& \hspace{.55cm} L_{\vec{n}}(x) = (-1)^{|\vec{n}|}\prod_{q=1}^p (\alpha_q+1)_{n_q} \operatorname{e}^x \,\pFq{p}{p}{\vec{\alpha}+\vec{n}+\vec{e}_p}{\vec{\alpha}+\vec{e}_p}{-x}.
\end{align}
%\end{subequations}
The type I polynomials were recently found in \cite{HS:JP-L1}:
\begin{align}
	\label{LaguerreITypeI}
	L^{(i)}_{\vec{n}}(x) &= (-1)^{|\vec{n}|-1}\dfrac{1}{(n_i-1)!\prod_{k=1, k \neq i}^{p} (\alpha_k-\alpha_i)_{n_k} \Gamma(\alpha_i+1)} \,\pFq{p}{p}{-n_i+1, (\alpha_i+1)\vec{e}_{p-1} - \vec{\alpha}^{\ast i} - \vec{n}^{\ast i}}{\alpha_i+1, (\alpha_i+1)\vec{e}_{p-1} - \vec{\alpha}^{\ast i}}{x} .
\end{align}

\subsubsection{Jacobi--Piñeiro Family}\label{S:Jacobi-Piñeiro}

Consider a system of Jacobi weights:
\begin{align*}
%\label{WeightsJP}
\begin{aligned}
 	w_{i}(x;\alpha_i,\beta)&=x^{\alpha_i}(1-x)^\beta, & i&\in\{1,\ldots,p\},& \Delta&=[0,1] .
 \end{aligned}
\end{align*}
Here \( \vec{\alpha}=(\alpha_1,\ldots,\alpha_p)\) with \( \alpha_1,\ldots,\alpha_p,\beta>-1\). To establish an AT system, we require that \( \alpha_i-\alpha_j\not\in\mathbb Z\) for \(i\neq j\).

The monic Jacobi--Piñeiro multiple orthogonal polynomials of type II admit both equivalent expressions, as shown in \cite[\S 3]{AskeyII}:
\begin{align}
\label{JPTypeII}
 P_{\vec{n}}(x)=&\begin{multlined}[t][.85\textwidth]
 	(-1)^{|\vec{n}|}\prod_{q=1}^p\dfrac{(\alpha_q+1)_{n_q}}{(\alpha_q+\beta+|\vec{n}|+1)_{n_q}}\sum_{l_1=0}^{n_1}\cdots\sum_{l_p=0}^{n_p}\,\prod_{q=1}^p
\dfrac{(-n_q)_{l_q}}{l_q!}\\\times 
 	\dfrac{(\alpha_1+n_1+1)_{l_2+\cdots+l_p}\cdots(\alpha_{p-1}+n_{p-1}+1)_{l_p}}{(\alpha_1+\beta+n_1+1)_{l_2+\cdots+l_p}\cdots(\alpha_{p-1}+\beta+n_1+\cdots+n_{p-1}+1)_{l_p}} \\\times \dfrac{(\alpha_1+\beta+n_1+1)_{l_1+\cdots+l_p}\cdots(\alpha_{p}+\beta+|\vec{n}|+1)_{l_p}}{(\alpha_1+1)_{l_1+\cdots+l_p}\cdots(\alpha_{p}+1)_{l_p}}\, x^{l_1+\cdots+l_p},
 \end{multlined}\\
\label{JPTypeIIWeighted}
 P_{\vec{n}}(x)=&(-1)^{|\vec{n}|}\prod_{q=1}^p\dfrac{(\alpha_q+1)_{n_q}}{(\alpha_q+\beta+|\vec{n}|+1)_{n_q}}(1-x)^{-\beta}\,\pFq{p+1}{p}{-|\vec{n}|-\beta,\vec{\alpha}+\vec{n}+\vec{e}_p}{\vec{\alpha}+\vec{e}_p}{x}.
\end{align}
In \cite{HS:JP-L1}, it was proved that the type I Jacobi--Piñeiro multiple orthogonal polynomials are, for \( i \in \{ 1, \ldots, p \} \) given by:
\begin{multline}
\label{JPTypeI}
 P^{(i)}_{\vec{n}} (x)
 =(-1)^{|\vec{n}|-1}\dfrac{\prod_{k=1}^{p}(\alpha_k+\beta+|\vec{n}|)_{n_k}}{(n_i-1)!\prod_{k=1,k\neq i}^{p}(\alpha_k-\alpha_i)_{n_k}}
\dfrac{\Gamma(\alpha_i+\beta+|\vec{n}|)}{\Gamma(\beta+|\vec{n}|)\Gamma(\alpha_i+1)}
 \\\times \pFq{p+1}{p}{-n_i+1,\alpha_i+\beta+|\vec{n}|, (\alpha_i+1)\vec{e}_{p-1}-\vec{\alpha}^{\ast i}-\vec{n}^{\ast i}}{\alpha_i+1,(\alpha_i+1)\vec{e}_{p-1}-\vec{\alpha}^{\ast i}}{x}.
\end{multline}

\subsubsection{Hahn Family}\label{S:Hahn_0}

This family arises from considering a system of discrete weights of the form:
\begin{align}
	\label{WeightsHahn}
\begin{aligned}
		w_{i}(x; \alpha_i, \beta, N) &= \frac{\Gamma(\alpha_i + x + 1)}{\Gamma(\alpha_i + 1) \Gamma(x + 1)} \frac{\Gamma(\beta + N - x + 1)}{\Gamma(\beta + 1) \Gamma(N - x + 1)}, & i & \in \{1, \ldots, p\}, & \Delta &= \{0, \ldots, N\}.
\end{aligned}
\end{align}
Here, \(\vec{\alpha} = (\alpha_1, \ldots, \alpha_p)\) with \(\alpha_1, \ldots, \alpha_p, \beta > -1\), \( |\vec{n}| \leq N \in \mathbb{N} \), and to establish an AT system, we require that \(\alpha_i - \alpha_j \notin \mathbb{Z}\) for \(i \neq j\).
The corresponding monic type II polynomials were first found for \(p = 2\) in \cite{Arvesu} and then generalized for \(p \geq 2\) in \cite{AskeyII}. 
These are given by:
\begin{multline}
	\label{HahnTypeII}
	Q_{\vec{n}}(x) = \prod_{q=1}^p \dfrac{(\alpha_q + 1)_{n_q}}{(\alpha_q + \beta + |\vec{n}| + 1)_{n_q}} 
	\\
\times\sum_{l_1 = 0}^{n_1} \cdots \sum_{l_p = 0}^{n_p} \prod_{q=1}^p \dfrac{(-n_q)_{l_q}}{l_q!} \frac{(-N)_{|\vec{n}|}}{(-N)_{l_1 + \cdots + l_p}} 
	\dfrac{(\alpha_1 + \beta + n_1 + 1)_{l_1 + \cdots + l_p} \cdots (\alpha_p + \beta + |\vec{n}| + 1)_{l_p}}{(\alpha_1 + 1)_{l_1 + \cdots + l_p} \cdots (\alpha_p + 1)_{l_p}} \\
	\times \dfrac{(\alpha_1 + n_1 + 1)_{l_2 + \cdots + l_p} \cdots (\alpha_{p-1} + n_{p-1} + 1)_{l_p}}{(\alpha_1 + \beta + n_1 + 1)_{l_2 + \cdots + l_p} \cdots (\alpha_{p-1} + \beta + n_1 + \cdots + n_{p-1} + 1)_{l_p}} (-x)_{l_1 + \cdots + l_p}.
\end{multline}
In \cite[\S 2]{HahnI}, it was proven that for \(p = 2\), \(i \in \{1, 2\}\), the type I polynomials can be expressed through a Kampé de Fériet function as:
\begin{multline}
	\label{HahnTypeI2measures}
	Q_{(n_1, n_2)}^{(i)}(x) %&
	 = 
	%\begin{multlined}[t][.8\textwidth]
		\frac{(-1)^{n_i - 1} (N + 1 - n_1 - n_2)! (n_1 + n_2 - 2)!}{(n_1 - 1)! (n_2 - 1)! (\beta + 1)_{n_1 + n_2 - 1} (\alpha_i + \beta + n_1 + n_2 + n_i)_{N + 1 - n_1 - n_2}} \frac{(\hat{\alpha}_i + \beta + \hat{n}_i + 1)_{n_1 + n_2 - 1}}{(\alpha_i - \hat{\alpha}_i - \hat{n}_i + 1)_{n_1 + n_2 - 1}} \\
		\times \KF{2:1;3}{2:0;2}{(-n_i + 1, -N): (\hat{\alpha}_i - \alpha_i - n_i + 1); (\alpha_i + \beta + n_1 + n_2, \alpha_i - \hat{\alpha}_i - \hat{n}_i + 1, -x)}{(-n_1 - n_2 + 2, \hat{\alpha}_i + \beta + \hat{n}_i + 1): --; (\alpha_i + 1, -N)}{1, 1} ,
	%\end{multlined}
\end{multline}
where \(\hat{\alpha}_i \coloneq \delta_{i, 2} \alpha_1 + \delta_{i, 1} \alpha_2\) and \(\hat{n}_i \coloneq \delta_{i, 2} n_1 + \delta_{i, 1} n_2\).

%%%%%%%%%%%%%%%%%%%%%%%%%%%%%%
\section{Hypergeometric Representations: Hahn Type I Polynomials}\label{S:Hahn}

The main aim of this section is to generalize the previous result of the Hahn type I polynomials for an arbitrary number of weight functions. To achieve this, we will reformulate the previous expression in a more convenient way. From the definition of the Kampé de Fériet series \eqref{KF}, expression \eqref{HahnTypeI2measures} can be written as
\begin{multline*}
	Q_{(n_1, n_2)}^{(i)}(x) = 
	%\begin{multlined}[t][.7\textwidth]
		\frac{(-1)^{n_i - 1} (N + 1 - n_1 - n_2)! (n_1 + n_2 - 2)!}{(n_1 - 1)! (n_2 - 1)! (\beta + 1)_{n_1 + n_2 - 1} (\alpha_i + \beta + n_1 + n_2 + n_i)_{N + 1 - n_1 - n_2}} \frac{(\hat{\alpha}_i + \beta + \hat{n}_i + 1)_{n_1 + n_2 - 1}}{(\alpha_i - \hat{\alpha}_i - \hat{n}_i + 1)_{n_1 + n_2 - 1}} \\
		\times \sum_{l = 0}^{n_i - 1} \frac{(-n_i + 1)_l (\alpha_i + \beta + n_1 + n_2)_l (\alpha_i - \hat{\alpha}_i - \hat{n}_i + 1)_l (-x)_l}{l! (-n_1 - n_2 + 2)_l (\hat{\alpha}_i + \beta + \hat{n}_i + 1)_l (\alpha_i + 1)_l} \, \pFq{3}{2}{-n_i + 1 + l, -N + l, \hat{\alpha}_i - \alpha_i - n_i + 1}{-n_1 - n_2 + 2 + l, \hat{\alpha}_i + \beta + \hat{n}_i + 1 + l}{1}.
	%\end{multlined}
\end{multline*}
Now, the previous \({}_3F_2\) function can be rewritten using Kummer's formula \eqref{Kummer} as follows:
\begin{multline*}
	\pFq{3}{2}{-n_i + 1 + l, -N + l, \hat{\alpha}_i - \alpha_i - n_i + 1}{-n_1 - n_2 + 2 + l, \hat{\alpha}_i + \beta + \hat{n}_i + 1 + l}{1} \\
	= \frac{(\alpha_i + \beta + N + 2 + l)_{n_i - 1 - l}}{(\hat{\alpha}_i + \beta + \hat{n}_i + 1 + l)_{n_i - 1 - l}} \, \pFq{3}{2}{-n_i + 1 + l, -n_1 - n_2 + 2 + N, \alpha_i - \hat{\alpha}_i - \hat{n}_i + 1 + l}{-n_1 - n_2 + 2 + l, \alpha_i + \beta + N + 2 + l}{1}.
\end{multline*}
By replacing this new \({}_3F_2\) function into the previous expression and simplifying, we find that the Hahn polynomials can be written as:
\begin{multline*}
	Q_{(n_1, n_2)}^{(i)}(x) = \dfrac{(-1)^{n_i - 1} (N + 1 - n_1 - n_2)! (n_1 + n_2 - 2)!
	(\alpha_1 + \beta + n_1 + n_2)_{n_1} (\alpha_2 + \beta + n_1 + n_2)_{n_2}
	}{
	(n_1 - 1)! (n_2 - 1)! (\beta + 1)_{n_1 + n_2 - 1} (\alpha_i + \beta + n_1 + n_2)_{N + 2 - n_1 - n_2}
	(\alpha_i - \hat{\alpha}_i - \hat{n}_i + 1)_{n_1 + n_2 - 1}} \\
%\dfrac{(\alpha_1 + \beta + n_1 + n_2)_{n_1} (\alpha_2 + \beta + n_1 + n_2)_{n_2}}{(\alpha_i - \hat{\alpha}_i - \hat{n}_i + 1)_{n_1 + n_2 - 1}} \\
	\times \underbrace{\KF{2:1;2}{2:0;1}{(-n_i + 1, \alpha_i - \hat{\alpha}_i - \hat{n}_i + 1) : (-n_1 - n_2 + 2 + N); (\alpha_i + \beta + n_1 + n_2, -x)}{(-n_1 - n_2 + 2, \alpha_i + \beta + N + 2): --; (\alpha_i + 1)}{1, 1}}_{\dfrac{(\alpha_i - \hat{\alpha}_i + 1)_{n_i - 1}}{(\hat{n}_i)_{n_i - 1}} \, \pFq{4}{3}{-n_i + 1, \alpha_i + \beta + n_1 + n_2, \alpha_i - \hat{\alpha}_i - \hat{n}_i + 1, x + \alpha_i + 1}{\alpha_i + 1, \alpha_i - \hat{\alpha}_i + 1, \alpha_i + \beta + N + 2}{1} \text{ by } \eqref{RR}}.
\end{multline*}
Simplifying, we arrive at the following formula:
\begin{multline*}
	Q_{(n_1, n_2)}^{(i)}(x) = \dfrac{(-1)^{n_1 + n_2 - 1} (N + 1 - n_1 - n_2)!}{(n_i - 1)! (\beta + 1)_{n_1 + n_2 - 1} (\alpha_i + \beta + n_1 + n_2)_{N + 2 - n_1 - n_2}} \dfrac{(\alpha_1 + \beta + n_1 + n_2)_{n_1} (\alpha_2 + \beta + n_1 + n_2)_{n_2}}{(\hat{\alpha}_i - \alpha_i)_{\hat{n}_i}} \\
	\times \pFq{4}{3}{-n_i + 1, \alpha_i + \beta + n_1 + n_2, \alpha_i - \hat{\alpha}_i - \hat{n}_i + 1, x + \alpha_i + 1}{\alpha_i + 1, \alpha_i - \hat{\alpha}_i + 1, \alpha_i + \beta + N + 2}{1}.
\end{multline*}
This simplified formula is easier to generalize than \eqref{HahnTypeI2measures}, leading to the desired generalization of the Hahn type I polynomials for an arbitrary number of weight functions.

\begin{teo}
 \label{HahnTypeITheorem}
The Hahn multiple orthogonal polynomials of type I for an arbitrary number 
\( p\) of weights are expressed~as:
 \begin{align}
 \label{HahnTypeI}
 Q_{\vec{n}}^{(i)}(x)&=\begin{multlined}[t][.8\textwidth]
 \dfrac{(-1)^{|\vec{n}|-1}(N+1-|\vec{n}|)!}{(n_i-1)!(\beta+1)_{|\vec{n}|-1}(\alpha_i+\beta+|\vec{n}|)_{N+2-|\vec{n}|}}\dfrac{\prod_{k=1}^p({\alpha}_k+\beta+|\vec{n}|)_{n_k}}{\prod_{k=1,k\neq i}^p({\alpha}_k-{\alpha}_i)_{{n}_k}}\\
 \times\pFq{p+2}{p+1}{-n_i+1,\alpha_i+\beta+|\vec{n}|,(\alpha_i+1)\vec{e}_{p-1}-\vec{\alpha}^{\ast i}-\vec{n}^{\ast i},x+\alpha_i+1}{\alpha_i+1,(\alpha_i+1)\vec{e}_{p-1}-\vec{\alpha}^{\ast i},\alpha_i+\beta+N+2}{1}
 \end{multlined}\\
 & =\sum_{l=0}^{n_i-1}C_{\vec{n}}^{(i),l}\,(x+\alpha_i+1)_l,
 \end{align}
 with
 \begin{multline}
 \label{coefHahnTypeI}
 C_{\vec{n}}^{(i),l}\coloneq\dfrac{(-1)^{|\vec{n}|-1}(N+1-|\vec{n}|)!}{(n_i-1)!(\beta+1)_{|\vec{n}|-1}(\alpha_i+\beta+|\vec{n}|)_{N+2-|\vec{n}|}}
 \\
 \times
 \dfrac{\prod_{k=1}^p({\alpha}_k+\beta+|\vec{n}|)_{n_k}}{\prod_{k=1,k\neq i}^p({\alpha}_k-{\alpha}_i)_{{n}_k}}
 \dfrac{(-n_i+1)_l(\alpha_i+\beta+|\vec{n}|)_l}{l!(\alpha_i+1)_l(\alpha_i+\beta+N+2)_l}
 \prod_{k=1,k\neq i}^p\dfrac{(\alpha_i-{\alpha}_k-{n}_k+1)_l}{(\alpha_i-{\alpha}_k+1)_l}.
 \end{multline}
\end{teo}

%\section{Proof using hypergeometric polynomials approach}

\begin{proof}
We need to demonstrate that the aforementioned polynomials satisfy the orthogonality relations \eqref{ortogonalidadTipoIDiscreta} with respect to the weight functions \eqref{WeightsHahn}.
However, we will 
%utilize 
use the Pochhammer basis \(\{(\beta + N - k + 1)_j\}\) instead of the power basis \(\{k^j\}\), as it proves to be more convenient, as will be shown. First, let us find a more suitable expression for the discrete integral
 \begin{align*}
 \sum_{k=0}^{N} (\beta+N-k+1)_j\,Q_{\vec{n}}^{(i)}(k)\,w_i(k) 
 & =\sum_{l=0}^{n_i-1}C_{\vec{n}}^{(i),l}\sum_{k=0}^N \dfrac{(\alpha_i+k+1)_l\Gamma(\alpha_i+k+1)}{\Gamma(\alpha_i+1)\Gamma(k+1)}\dfrac{(\beta+N-k+1)_j\Gamma(\beta+N-k+1)}{\Gamma(\beta+1)\Gamma(N-k+1)}\\
 &=\dfrac{(\beta+1)_j}{N!}\sum_{l=0}^{n_i-1}C_{\vec{n}}^{(i),l}{(\alpha_i+1)_l}\underbrace{\sum_{k=0}^N \dbinom{N}{k}(\alpha_i+1+l)_{k}(\beta+1+j)_{N-k}}_{(\alpha_i+\beta+2+j+l)_N} .
 \end{align*}
The expression under the braces can be easily deduced using \eqref{Chu-Van}.
Replacing the coefficients from \eqref{coefHahnTypeI} and simplifying, we find that
 \begin{multline*}
 \sum_{k=0}^{N} (\beta+N-k+1)_j\,Q_{\vec{n}}^{(i)}(k)\,w_i(k)\\
 =\dfrac{(-1)^{|\vec{n}|-1}(N+1-|\vec{n}|)!\prod_{k=1}^p({\alpha}_k+\beta+|\vec{n}|)_{n_k}}{N!(\beta+1+j)_{|\vec{n}|-1-j}}\dfrac{\Gamma(\alpha_i+\beta+|\vec{n}|)(\alpha_i+\beta+N+2)_{j}}{(n_i-1)!\Gamma(\alpha_i+\beta+2+j)\prod_{k=1,k\neq i}^p({\alpha}_k-{\alpha}_i)_{{n}_k}}\\
 \times\pFq{p+2}{p+1}{-n_i+1,\alpha_i+\beta+N+2+j,\alpha_i+\beta+|\vec{n}|,(\alpha_i+1)\vec{e}_{p-1}-\vec{\alpha}^{\ast i}-\vec{n}^{\ast i}}{\alpha_i+\beta+N+2,\alpha_i+\beta+2+j,(\alpha_i+1)\vec{e}_{p-1}-\vec{\alpha}^{\ast i}}{1} .
 \end{multline*}
In this scenario, the orthogonality conditions \eqref{ortogonalidadTipoIDiscreta} that we need to establish are equivalent to the following hypergeometric summation formula:
 \begin{multline}
 \label{orthogonalityHahnTypeI}
 \dfrac{(-1)^{|\vec{n}|-1}(N+1-|\vec{n}|)!\prod_{k=1}^p({\alpha}_k+\beta+|\vec{n}|)_{n_k}}{N!(\beta+1+j)_{|\vec{n}|-1-j}}
 \sum_{i=1}^p\dfrac{\Gamma(\alpha_i+\beta+|\vec{n}|)(\alpha_i+\beta+N+2)_{j}}{(n_i-1)!\Gamma(\alpha_i+\beta+2+j)\prod_{k=1,k\neq i}^p({\alpha}_k-{\alpha}_i)_{{n}_k}}\\
 \times\pFq{p+2}{p+1}{-n_i+1,\alpha_i+\beta+N+2+j,\alpha_i+\beta+|\vec{n}|,(\alpha_i+1)\vec{e}_{p-1}-\vec{\alpha}^{\ast i}-\vec{n}^{\ast i}}{\alpha_i+\beta+N+2,\alpha_i+\beta+2+j,(\alpha_i+1)\vec{e}_{p-1}-\vec{\alpha}^{\ast i}}{1}
 \\
 =\begin{cases}
 0,\;\text{if}\; j\in\{0,\ldots,|\vec{n}|-2\},\\
 (-1)^{|n|-1},\;\text{if}\; j=|\vec{n}|-1.
 \end{cases}
 \end{multline}
The value of the inner product with respect to the highest Pochhammer symbol must be \( (-1)^{|\vec{n}|-1}\) because \((\beta+N-k+1)_{|\vec{n}|-1} = (-1)^{|\vec{n}|-1} k^{|\vec{n}|-1} + \cdots\).
 
Without loss of generality, we will now consider the summand \(i=1\) from the previous sum and apply equation \eqref{KPTheoremReformulated} to it. We need to distinguish between two cases depending on \(j\).
If \(j\in\{0,\ldots,|\vec{n}|-2\}\), in accordance with the notation introduced in Lemma \ref{lemma:KP}, we can identify the following parameters:
\begin{align*}
 &a = -n_1 + 1, 
 \quad (f_1, f_2) = (\alpha_1 + \beta + N + 2, \alpha_1 + \beta + 2 + j), 
 \quad
 (m_1, m_2) = (j, |\vec{n}| - 2 - j), \\
 &(b_1, \ldots, b_{p-1}) = (\alpha_1 - \alpha_2 - n_2 + 1, \ldots, \alpha_1 - \alpha_p - n_p + 1), \quad
 (k_1, \ldots, k_{p-1}) = (n_2, \ldots, n_p).
\end{align*}
The condition \(\Re(k_1 + \cdots + k_{p-1} - a - m_1 - m_2) = 1 > 0\) is met, allowing us to apply equation \eqref{KPTheoremReformulated}. 
After some calculations, we can express the Hahn multiple orthogonal polynomials of type I for \(i=1\) and \(j \in \{0, \ldots, |\vec{n}|-2\}\) as follows:
 \begin{multline*}
 \dfrac{\Gamma(\alpha_1+\beta+|\vec{n}|)(\alpha_1+\beta+N+2)_{j}}{(n_1-1)!\Gamma(\alpha_1+\beta+2+j)\prod_{k=2}^p({\alpha}_k-{\alpha}_1)_{{n}_k}} 
  \\
\times \pFq{p+2}{p+1}{-n_1+1,\alpha_1+\beta+N+2+j,\alpha_1+\beta+|\vec{n}|,(\alpha_1+1)\vec{e}_{p-1}-\vec{\alpha}^{\ast 1}-\vec{n}^{\ast 1}}{\alpha_1+\beta+N+2,\alpha_1+\beta+2+j,(\alpha_1+1)\vec{e}_{p-1}-\vec{\alpha}^{\ast 1}}{1}\\
 =-\sum_{i=2}^p\dfrac{\Gamma(\alpha_i+\beta+|\vec{n}|)(\alpha_i+\beta+N+2)_{j}}{(n_i-1)!\Gamma(\alpha_i+\beta+2+j)\prod_{k=1,k\neq i}^p({\alpha}_k-{\alpha}_i)_{{n}_k}}
 \\ \times 
 \pFq{p+2}{p+1}{-n_i+1,\alpha_i+\beta+N+2+j,\alpha_i+\beta+|\vec{n}|,(\alpha_i+1)\vec{e}_{p-1}-\vec{\alpha}^{\ast i}-\vec{n}^{\ast i}}{\alpha_i+\beta+N+2,\alpha_i+\beta+2+j,(\alpha_i+1)\vec{e}_{p-1}-\vec{\alpha}^{\ast i}}{1} ,
 \end{multline*}
 which is equivalent to say that the left-hand side of equation \eqref{orthogonalityHahnTypeI} equals 0 for \( j=0,\ldots,|\vec{n}|-2\). 
 
 Now, when \( j=|\vec{n}|-1\), referring to the notation introduced in Lemma \ref{lemma:KP}, we can identify:
 \begin{align*}
 &a=-n_1+1; \quad f=\alpha_1+\beta+N+2, \quad m=|\vec{n}|-1;\\
 &(b_1,\ldots,b_{p})=(\alpha_1+\beta+|\vec{n}|,\alpha_1-\alpha_2-n_2+1,\ldots,\alpha_1-\alpha_p-n_p+1),\quad (k_1,\ldots,k_{p})=(1,n_2,\ldots,n_p) .
 \end{align*}
 The condition \( \Re(k_1+\cdots+k_{p}-a-m)=1>0\) is met, allowing us to apply equation \eqref{KPTheoremReformulated}. This time we introduce an additional parameter \( b\), which results in an extra term in the decomposition. After some calculations, we can express
 \begin{multline*}
 \dfrac{(-1)^{|\vec{n}|-1}(N+1-|\vec{n}|)!\prod_{k=1}^p({\alpha}_k+\beta+|\vec{n}|)_{n_k}}{N!}\dfrac{\Gamma(\alpha_1+\beta+|\vec{n}|)(\alpha_1+\beta+N+2)_{|\vec{n}|-1}}{(n_1-1)!\Gamma(\alpha_1+\beta+|\vec{n}|+1)\prod_{k=2}^p({\alpha}_k-{\alpha}_1)_{{n}_k}}\\
 \times\pFq{p+2}{p+1}{-n_1+1,\alpha_1+\beta+N+|\vec{n}|+1,\alpha_1+\beta+|\vec{n}|,(\alpha_1+1)\vec{e}_{p-1}-\vec{\alpha}^{\ast 1}-\vec{n}^{\ast 1}}{\alpha_1+\beta+N+2,\alpha_1+\beta+|\vec{n}|+1,(\alpha_1+1)\vec{e}_{p-1}-\vec{\alpha}^{\ast 1}}{1}\\
 =(-1)^{|\vec{n}|-1}-\dfrac{(-1)^{|\vec{n}|-1}(N+1-|\vec{n}|)!\prod_{k=1}^p({\alpha}_k+\beta+|\vec{n}|)_{n_k}}{N!}\sum_{i=2}^p\dfrac{\Gamma(\alpha_i+\beta+|\vec{n}|)(\alpha_i+\beta+N+2)_{|\vec{n}|-1}}{(n_i-1)!\Gamma(\alpha_i+\beta+|\vec{n}|+1)\prod_{k=1,k\neq i}^p({\alpha}_k-{\alpha}_i)_{{n}_k}}\\
 \times\pFq{p+2}{p+1}{-n_i+1,\alpha_i+\beta+N+|\vec{n}|+1,\alpha_i+\beta+|\vec{n}|,(\alpha_i+1)\vec{e}_{p-1}-\vec{\alpha}^{\ast i}-\vec{n}^{\ast i}}{\alpha_i+\beta+N+2,\alpha_i+\beta+|\vec{n}|+1,(\alpha_i+1)\vec{e}_{p-1}-\vec{\alpha}^{\ast i}}{1} ,
 \end{multline*}
which is equivalent to saying that the left-hand side of equation \eqref{orthogonalityHahnTypeI} equals \( (-1)^{|\vec{n}|-1}\) when \( j=|\vec{n}|-1\). 
\end{proof}

\section{Integral Representations}\label{S: Integral}

In this section, we will derive some contour integral representations for the Hahn multiple orthogonal polynomials of types I and II. Rather than starting from the hypergeometric representations in Theorem~\ref{HahnTypeITheorem} and in equation~\eqref{HahnTypeII}, we aim to provide a constructive proof for these integral representations based on the residue theorem and the Mellin transform. This approach will also lead to a constructive proof for the associated hypergeometric representations. In order to get a better understanding of the ideas behind the techniques presented here, we will first investigate the settings of the multiple Laguerre polynomials of the first kind and the Jacobi--Piñeiro polynomials.

\subsection{Laguerre of the First Kind Multiple Orthogonal Polynomials}

Here, we rely on \S \ref{S:Laguerre of the first kind}.

\subsubsection{Multiple Orthogonal Polynomials of Type I}

Our approach relies on encoding the type I linear forms \( L_{\vec{n}}^{(I)}(x) = \sum_{i=1}^p L_{\vec{n}}^{(i)}(x) x^{\alpha_i} \) as a contour integral in the following manner.

\begin{pro} \label{LI_PC}
 Let us consider:
\begin{enumerate}[\rm i)]
	\item a %polynomial combination 
linear form
\(
F_{\vec{n}}(x) = \sum_{i=1}^p A_{\vec{n},i}(x) x^{\alpha_i} \) with \( \deg A_{\vec{n},i} \leqslant n_i-1\),
\item 
an anticlockwise contour \( \Sigma\subset\C\) enclosing once \(\mathcal I=\cup_{i=1}^p[\alpha_i,n_i+\alpha_i-1]\),
\item and a function \( \varphi \) analytic inside and on \( \Sigma\) that doesn't vanish on \( \mathcal I\). 
\end{enumerate}
Then there exists a unique polynomial \( p_{\sz{n}-1}(t)\) of degree at most \( \sz{n}-1\) such that
\begin{align}\label{eq:integral_representation}
 F_{\vec{n}}(x) = \int_{\Sigma}\frac{p_{\sz{n}-1}(t)}{\prod_{i=1}^p (\alpha_i-t)_{n_i}} \varphi (t) x^t \frac{ \operatorname d t }{2 \pi \ii }.
 \end{align}
\end{pro}
\begin{proof}
Suppose that \( p_{|\vec{n}|-1}(t) \) is a polynomial of degree at most \( |\vec{n}|-1 \). Then, it follows from the residue theorem that
\begin{align*}
	\int_{\Sigma} \frac{p_{|\vec{n}|-1}(t)}{\prod_{i=1}^p (\alpha_i - t)_{n_i}} \varphi(t) x^t \frac{\operatorname dt}{2 \pi \operatorname i} = \sum_{i=1}^p \sum_{k=0}^{n_i-1} \frac{(-1)^kp_{|\vec{n}|-1}(k + \alpha_i) \varphi(k + \alpha_i) }{ (n_i-1-k)! k! \prod_{j=1,j\neq i}^p (\alpha_j-\alpha_i-k)_{n_j} i}x^{k + \alpha_i}.
\end{align*}
If we require the Equation
\eqref{eq:integral_representation}
to be satisfied, and we write \( A_{\vec{n},i}(x) = \sum_{k=0}^{n_i-1} A_{\vec{n},i}[k] x^k \), we conclude that
\begin{align*}
 A_{\vec{n},i}[k] = \frac{(-1)^kp_{|\vec{n}|-1}(k + \alpha_i) \varphi(k + \alpha_i) }{ (n_i-1-k)! k! \prod_{j=1,j\neq i}^p (\alpha_j-\alpha_i-k)_{n_j}}.
\end{align*}
These \( |\vec{n}| \) interpolation conditions then uniquely determine the polynomial \( p_{|\vec{n}|-1}(t) \) of degree \( |\vec{n}|-1 \).
\end{proof}

The orthogonality conditions, along with the normalization condition, will enable us to explicitly determine the underlying polynomials for the type I linear forms \( L_{\vec{n}}^{(I)}(x) = \sum_{i=1}^p L_{\vec{n}}^{(i)}(x) x^{\alpha_i} \). A crucial element in the proof will be the following lemma.

\begin{lemma} \label{IR_lem}
Let \( \Sigma\subset\C\) be an anticlockwise contour enclosing once \(\cup_{i=1}^p[\alpha_i,n_i+\alpha_i-1]\). Suppose that \(p_{\sz{n}-1}(t)\) is polynomial of degree at most \(\sz{n}-1\) for which
\begin{align}
	\int_\Sigma \frac{p_{\sz{n}-1}(t)}{\prod_{i=1}^p(\alpha_i-t)_{n_i}} q_{\sz{n}-2}(t) \frac{\operatorname{d} t}{2 \pi \ii} = 0,
\end{align}
for every polynomial \(q_{\sz{n}-2}(t)\) of degree at most \(\sz{n}-2\). Then \(p_{\sz{n}-1}(t)\) is a constant.
\end{lemma}
\begin{proof}
    Note that the assumption holds in particular for every polynomial of the form 
    \begin{align}
        q_{\sz{n}-2}(t) = \frac{\prod_{i=1}^p(\alpha_i-t)_{n_i}}{(\alpha_1-t)(\alpha_j+k-t)},    \end{align}
     where \(k\in\{0,\ldots,n_j-1\}\) and \(j\in\{1,\ldots,p\}\) with \(\alpha_j+k\neq \alpha_1\). After an application of the residue theorem, the identity
    \begin{align}
	\int_\Sigma \frac{p_{\sz{n}-1}(t)}{(\alpha_1-t)(\alpha_j+k-t)} \frac{\operatorname{d} t}{2 \pi \ii} = 0,
    \end{align}
    gives
    \begin{align}
 p_{\sz{n}-1}(\alpha_j + k) = p_{\sz{n}-1}(\alpha_1),
    \end{align}
    indicating that \( p_{\sz{n}-1}(t) \) remains constant across the \( \sz{n} \) zeros of \( \prod_{i=1}^p (\alpha_i - t)_{n_i} \). Since \( p_{\sz{n}-1}(t) \) has degree at most \( \sz{n} - 1 \), it follows that \( p_{\sz{n}-1}(t) \) is a constant.
\end{proof}

\begin{teo} \label{LI_IR}
Let \( \Sigma \subset \mathbb{C} \) be as in Proposition \ref{LI_PC} and take \( \operatorname{Re}(t) > -1 \). Then the linear form
\begin{align*}
L_{\vec{n}}^{(I)}(x) \coloneq \sum_{i=1}^p L_{\vec{n}}^{(i)}(x) x^{\alpha_i} ,
\end{align*}
is given by
\begin{align}\label{eq:LI}
	L_{\vec{n}}^{(I)}(x) = (-1)^{\sz{n}-1} \int_\Sigma \frac{1}{\Gamma(t+1)} \frac{x^t}{\prod_{i=1}^p(\alpha_i-t)_{n_i}} \frac{\operatorname{d} t}{2 \pi \ii}.
\end{align}
\end{teo}
\begin{proof}
According to Proposition \ref{LI_PC}, the linear form \( L_{\vec{n}}^{(I)}(x) \) is given by
\begin{align*}
 L_{\vec{n}}^{(I)}(x) = \int_{\Sigma} \frac{p_{\sz{n}-1}(t)}{\prod_{i=1}^p (\alpha_i - t)_{n_i}} \varphi(t) x^t \frac{\operatorname{d} t}{2 \pi \ii},
\end{align*}
where \( p_{\sz{n}-1}(t) \) is a polynomial of degree at most \( \sz{n} - 1 \), and \( \varphi \) is chosen below. Recalling \eqref{eq:Mellin_exp} and applying the weighted Mellin transform, we obtain
\begin{align*}
 \int_0^{\infty} L_{\vec{n}}^{(I)}(x) \operatorname{e}^{-x} x^s \operatorname{d} x = \int_\Sigma \frac{p_{\sz{n}-1}(t)}{\prod_{i=1}^p (\alpha_i - t)_{n_i}} \varphi(t) \Gamma(s + t + 1) \frac{\operatorname{d} t}{2 \pi \ii}.
\end{align*}
By choosing \( \varphi(t) = 1 / \Gamma(t + 1) \), for \( s \in \mathbb{N} \), a polynomial of degree \( s \) appears in the integrand. The orthogonality conditions \eqref{ortogonalidadTipoIContinua} then require
\begin{align*}
 \int_\Sigma \frac{p_{\sz{n}-1}(t)}{\prod_{i=1}^p (\alpha_i - t)_{n_i}} (t + 1)_s \frac{\operatorname{d} t}{2 \pi \ii} = 0, \quad s \in \{0, \ldots, \sz{n} - 2\},
\end{align*}
or equivalently,
\begin{align*}
 \int_\Sigma \frac{p_{\sz{n}-1}(t)}{\prod_{i=1}^p (\alpha_i - t)_{n_i}} q_{\sz{n}-2}(t) \frac{\operatorname{d} t}{2 \pi \ii} = 0,
\end{align*}
for every polynomial \( q_{\sz{n}-2}(t) \) of degree at most \( \sz{n} - 2 \). Hence, it follows from Lemma \ref{IR_lem} that \( p_{\sz{n}-1}(t)\) is a constant, say \( p_{\sz{n}-1}(t) = p_{\sz{n}-1}[0] \). The specific value is determined by the normalization condition
\begin{align*}
 \int_\Sigma \frac{p_{\sz{n}-1}[0]}{\prod_{j=1}^p (\alpha_i - t)_{n_i}} (t + 1)_{\sz{n} - 1} \frac{\operatorname{d} t}{2 \pi \ii} = 1.
\end{align*}
Since the integrand decays as \( \operatorname{O}(t^{-1}) \) as \( |t| \to \infty \), we evaluate the integral using the residue at infinity:
\begin{align*}
 \int_\Sigma \frac{p_{\sz{n}-1}[0]}{\prod_{i=1}^p (\alpha_i - t)_{n_i}} (t + 1)_{\sz{n} - 1} \frac{\operatorname{d} t}{2 \pi \ii} = - \lim_{|t| \to \infty} t \frac{p_{\sz{n}-1}[0]}{\prod_{j=1}^p (\alpha_i - t)_{n_i}} (t + 1)_{\sz{n} - 1} = (-1)^{\sz{n} - 1} p_{\sz{n}-1}[0].
\end{align*}
Thus, \( p_{\sz{n}-1}[0] = (-1)^{\sz{n} - 1} \).
\end{proof}

\begin{coro}
The computation of the contour integral \eqref{eq:LI}  leads to the identities given in \eqref{LaguerreITypeI}.
\end{coro}
\begin{proof}
 The integrand only has simple poles at \(t=\alpha_i+k\) for \(k\in\{0,\ldots,n_i-1\}\) and \(i\in\{1,\ldots,p\}\). The residue at \(t=\alpha_i+k\) is given by
 \begin{align*}
  \frac{1}{\Gamma(k+\alpha_i+1)} &\frac{(-1)^k}{k! (n_i-k-1)!} \frac{1}{\prod_{j=1,j\neq i}^p(\alpha_j-\alpha_i-k)_{n_j}} x^{k+\alpha_i} \\
  & = \frac{1}{(\alpha_i+1)_k \Gamma(\alpha_i+1)} \frac{(-n_i+1)_k}{k!(n_i-1)!} 
  \prod_{j=1,j\neq i}^p\frac{(\alpha_i-\alpha_j-n_j+1)_k}{ (\alpha_i-\alpha_j+1)_k(\alpha_j-\alpha_i)_{n_j}}
x^{k+\alpha_i}.
 \end{align*}
 Consequently,
 \begin{align*}
  L_{\vec{n}}^{(i)}[k] = \frac{(-1)^{\sz{n}-1}}{(n_i-1)! \prod_{j=1,j\neq i}^p(\alpha_j-\alpha_i)_{n_j} \Gamma(\alpha_i+1)} \frac{(-n_i+1)_k}{(\alpha_i+1)_k }   \prod_{j=1,j\neq i}^p\frac{(\alpha_i-\alpha_j-n_j+1)_k}{ (\alpha_i-\alpha_j+1)_k}\frac{1}{k!} ,
 \end{align*}
 which leads to \eqref{LaguerreITypeI}.
\end{proof}

\subsubsection{Multiple Orthogonal Polynomials of Type II}
Our approach for the type II polynomials in \S \ref{S:Laguerre of the first kind} is based on our ability to explicitly compute their weighted Mellin transform. The result below for \(p=2\) was previously found in \cite[Lemma 4.2]{VAWolfs}, using a different approach.

\begin{pro}
 The Mellin transform of the functions \(L_{\vec{n}}(x) \operatorname{e}^{-x}\) is given by
\begin{align*}
\int_0^\infty L_{\vec{n}}(x) \operatorname{e}^{-x} x^{s-1} \operatorname d x = (-1)^{\sz{n}} \Gamma(s) \prod_{i=1}^p (\alpha_i+1-s)_{n_i}.
\end{align*}
\end{pro}
\begin{proof}
For the LHS, Equation \eqref{eq:Mellin_exp} and the expression \(L_{\vec{n}}(x) = \sum_{k=0}^{\sz{n}} L_{\vec{n}}[k] x^k\) leads to
\begin{align*}
\sum_{k=0}^{\sz{n}} L_{\vec{n}}[k] \int_0^\infty x^{s+k-1} \operatorname{e}^{-x} \operatorname d x = \sum_{k=0}^{\sz{n}} L_{\vec{n}}[k] \Gamma(s+k).
\end{align*}
Let us take out the common factor \( \Gamma(s)\) to obtain
\begin{align*}
 \Gamma(s) \sum_{k=0}^{\sz{n}} L_{\vec{n}}[k] (s)_k = \Gamma(s) p_{\sz{n}}(s),
 \end{align*}
 in terms of some polynomial \(p_{\sz{n}}(s)\) of degree at most \( \sz{n}\). It then follows from the orthogonality conditions~\eqref{ortogonalidadTipoIIContinua} that \(p_{\sz{n}}(s)\) needs to vanish at the \( \sz{n}\) values \(k+\alpha_i\) for \(k\in\{1,\ldots,n_i\}\) and \( i\in\{1,\ldots,p\}\). Hence, \(p_{\sz{n}}(s)\) must be a scalar multiple of \( \prod_{i=1}^p (\alpha_i+1-s)_{n_i}\). The fact that \(L_{\vec{n}}(x)\) is monic, i.e. \(L_{\vec{n}}[\sz{n}]=1\), then implies that \(\displaystyle
 p_{\sz{n}}(s) = (-1)^{\sz{n}} \prod_{i=1}^p (\alpha_i+1-s)_{n_i} \).
\end{proof}

We employ the inverse Mellin transform to derive the following integral representation for the weighted type II polynomials. Notice the similarities with the contour integral used for the type I linear form in Theorem~\ref{LI_IR}.

\begin{teo} \label{LII_IMT}
For \(x\in (0,\infty)\), the function \(L_{\vec{n}}(x) \operatorname{e}^{-x}\) admits the following integral representation:
\begin{align}\label{eq:LII}
 L_{\vec{n}}(x) \operatorname{e}^{-x} = (-1)^{\sz{n}} \int_{\mathcal{C}} \Gamma(s) \prod_{i=1}^p (\alpha_i+1-s)_{n_i} x^{-s} \frac{ \operatorname d s }{2 \pi \ii },
 \end{align}
 where \( \mathcal{C}\) is an anticlockwise contour enclosing \( (-\infty,0]\) once. 
\end{teo}

\begin{coro}
The explicit hypergeometric expression in \eqref{LaguerreFKTypeIIWeighted} follows from the evaluation  of the contour integral  \eqref{eq:LII}.
\end{coro}
\begin{proof}
 The integrand only has simple poles at \(s=-k\) for \( k\in \mathbb{N}_0\). The residue at \(t=-k\) is given by
 \begin{align*}
  \frac{(-1)^k}{k!} \prod_{i=1}^p (\alpha_i+1+k)_{n_i} x^k = 
  \prod_{i=1}^p 
  \frac{(\alpha_i+1)_{n_i}  (n_i+\alpha_i+1)_k}{(\alpha_i+1)_k}  
  \frac{(-x)^k}{k!}.
 \end{align*}
 To establish \eqref{LaguerreFKTypeIIWeighted}, it then remains to note that we may replace the contour integral \(\int_{\mathcal{C}}\) by \(\lim_{M\to\infty} \int_{\mathcal{C}_M}\) in terms of anticlockwise oriented rectangles \(\mathcal{C}_M\) that connect the vertices \(-M\pm \operatorname i/2\) and \(1/2\pm \operatorname i/2\). We may do this because Stirling's formula \eqref{eq:Stirling} implies there exists  $K>0$ such  that \(|\Gamma(-M+ \operatorname i s)| \sim K M^{-M-1/2}\operatorname e^{M} \) as \(M\to\infty \) %(recall that  $f(M) \asymp g(M)$ as $M\to\infty $ means that $\lim_{M\to\infty}f(M)/g(M)\in\R\setminus\{0\}$) 
 and thus
 \begin{align*}
  \lim_{M\to\infty} \int_{-M-\frac{\operatorname i}{2}}^{-M+\frac{\operatorname i}{2}} \Gamma(s) \prod_{i=1}^p (\alpha_i+1-s)_{n_i} x^{-s} \frac{ \operatorname d s }{2 \pi \ii } = 0.
 \end{align*}
 
% - The notation f(M) \asymp g(M) as M\to\infty means that \lim_{M\to\infty}f(M)/g(M)\in\R\backslash\{0\}."
% - Proof of Thm 3.7, 1st line after 2nd formula: Here we don't really use (9), but rather that the Mellin transform of a beta density is
% (\mathcal{M}(1-x)^\beta)(s) = \Gamma(s)\Gamma(\beta+1)/\Gamma(s+\beta+1).
\end{proof}

\subsection{Jacobi--Piñeiro Multiple Orthogonal Polynomials}

Now we follow \S \ref{S:Jacobi-Piñeiro}.

\subsubsection{Multiple Orthogonal Polynomials of Type I}

The counterpart of Theorem \ref{LI_IR} in this context is as follows.

\begin{teo} \label{JI_IR}
Let \( \Sigma \subset \mathbb{C} \) be as in Proposition \ref{LI_PC} and take \( \operatorname{Re}(t) > -1 \). Then the linear form \begin{align*}
P_{\vec{n}}^{(I)}(x)\coloneq\sum_{i=1}^p P_{\vec{n}}^{(i)}(x) x^{\alpha_i}
\end{align*}
 is given by
\begin{align}\label{eq:JPI}
 P_{\vec{n}}^{(I)}(x)= (-1)^{\sz{n}-1} \frac{\prod_{i=1}^p (\alpha_i+\beta+\sz{n})_{n_i}}{\Gamma(\sz{n}+\beta)} \int_\Sigma \frac{\Gamma(t+\beta+\sz{n})}{\Gamma(t+1)} \frac{x^t}{\prod_{i=1}^p(\alpha_i-t)_{n_i}} \frac{ \operatorname d t }{2 \pi \ii }.
 \end{align}
\end{teo}
\begin{proof} 
Using Proposition \ref{LI_PC}, we express \(P_{\vec{n}}^{(I)}(x)\) as
\begin{align*}
P_{\vec{n}}^{(I)}(x) = \int_{\Sigma}\frac{p_{\sz{n}-1}(t)}{\prod_{i=1}^p (\alpha_i-t)_{n_i}} \varphi (t) x^t \frac{ \operatorname d t }{2 \pi \ii },
\end{align*}
where \(p_{\sz{n}-1}(t)\) is a polynomial of degree at most \( \sz{n}-1\) and \(\varphi(t)\) is chosen below. After performing the weighted Mellin transform, we obtain
\begin{align*}
\int_0^1 P_{\vec{n}}^{(I)}(x) (1-x)^{\beta} x^s \operatorname d x = \int_\Sigma\frac{p_{\sz{n}-1}(t)}{\prod_{i=1}^p (\alpha_i-t)_{n_i}} \varphi (t) \frac{\Gamma(s+t+1)\Gamma(\beta+1)}{\Gamma(s+t+\beta+2)} \frac{ \operatorname d t }{2 \pi \ii },
\end{align*}
were \eqref{eq:Mellin_beta} has been used.

The orthogonality conditions \eqref{ortogonalidadTipoIContinua} are then given by
\begin{align*}
\int_\Sigma\frac{p_{\sz{n}-1}(t)}{\prod_{i=1}^p (\alpha_i-t)_{n_i}} \varphi (t) \frac{\Gamma(s+t+1)\Gamma(\beta+1)}{\Gamma(s+t+\beta+2)} \frac{ \operatorname d t }{2 \pi \ii } = 0,\quad s\in\{0,\ldots,\sz{n}-2\}.
\end{align*}
Setting \( \varphi (t) = \Gamma(t+\beta+\sz{n})/(\Gamma(t+1)\Gamma(\beta+1)) \) is convenient here, because for \( s\in \{0,\ldots,\sz{n}-2\} \) the integrand involves the polynomial \((t+1)_s (s+t+\beta+2)_{\sz{n}-2-s}\). The conditions for these integrals to vanish then require
\begin{align*}
\int_\Sigma \frac{p_{\sz{n}-1}(t)}{\prod_{i=1}^p (\alpha_i-t)_{n_i}} q_{\sz{n}-2}(t) \frac{ \operatorname d t }{2 \pi \ii } = 0,
\end{align*}
for every polynomial \(q_{\sz{n}-2}(t)\) of degree at most \( \sz{n}-2\). In that case, Lemma \ref{IR_lem} implies that \(p_{\sz{n}-1}(t)\) is a constant, say \(p_{\sz{n}-1}(t) = p_{\sz{n}-1}[0]\). The precise value of \(p_{\sz{n}-1}[0]\) follows from the normalization condition
\begin{align*}
\int_\Sigma \frac{p_{\sz{n}-1}[0]}{\prod_{i=1}^p (\alpha_i-t)_{n_i}} \frac{(t+1)_{\sz{n}-1}}{t+\sz{n}+\beta} \frac{ \operatorname d t }{2 \pi \ii } = 1.
\end{align*}
Since the integrand behaves like
\(\operatorname{O}(|t|^{-2})\) as \( |t|\to\infty \), we evaluate the integral using the residue at \( t=-\sz{n}-\beta \) (note that $-\sz{n}-\beta<-1$ for \( \sz{n}\geq 2 \)):
\begin{align*}
\int_\Sigma \frac{p_{\sz{n}-1}[0]}{\prod_{i=1}^p (\alpha_i-t)_{n_i}} \frac{(t+1)_{\sz{n}-1}}{t+\sz{n}+\beta} \frac{ \operatorname d t }{2 \pi \ii } = \frac{p_{\sz{n}-1}[0]}{\prod_{i=1}^p (\alpha_i+\beta+\sz{n})_{n_i}} (-\sz{n}-\beta+1)_{\sz{n}-1}.
\end{align*}
Thus, \(\displaystyle
p_{\sz{n}-1}[0] = (-1)^{\sz{n}-1} \frac{\prod_{i=1}^p (\alpha_i+\beta+\sz{n})_{n_i}}{(\beta+1)_{\sz{n}-1}}\).
\end{proof}

\begin{coro}
	The explicit expressions \eqref{JPTypeI} are recovered from the evaluation of the contour integral \eqref{eq:JPI}.
\end{coro}
\begin{proof}
 The integrand only has simple poles inside the contour \( \Sigma\) at \(t=\alpha_i+k\) for \(k=0,\ldots,n_i-1\) and \(i=1,\ldots,p\). Since \(\text{Re}(t)>-1\), the poles of \(\Gamma(t+\beta+\sz{n})\) are outside of \( \Sigma\) (assuming that \( \sz{n}\geq 2 \)). The residue at \(t=\alpha_i+k\) is given by
 \begin{multline*}
  \frac{\Gamma(k+\alpha_i+\beta+\sz{n})}{\Gamma(k+\alpha_i+1)} \frac{(-1)^k}{k! (n_i-k-1)!} \frac{1}{\prod_{j=1,j\neq i}^p(\alpha_j-\alpha_i-k)_{n_j}} x^{k+\alpha_i} \\
   \quad = \frac{(\sz{n}+\alpha_i+\beta)_k \Gamma(\sz{n}+\alpha_i+\beta)}{(\alpha_i+1)_k \Gamma(\alpha_i+1)} \frac{(-n_i+1)_k}{k! (n_i-1)!} \frac{1}{(\alpha_i+1)_k \Gamma(\alpha_i+1)} \frac{(-n_i+1)_k}{k!(n_i-1)!} 
  \\\times  \prod_{j=1,j\neq i}^p\frac{(\alpha_i-\alpha_j-n_j+1)_k}{ (\alpha_i-\alpha_j+1)_k(\alpha_j-\alpha_i)_{n_j}}x^{k+\alpha_i}.
 \end{multline*}
 Consequently,
 \begin{multline*}
  P_{\vec{n}}^{(i)}[k] = \frac{(-1)^{\sz{n}-1} \prod_{j=1}^p (\alpha_j+\beta+\sz{n})_{n_j} \Gamma(\sz{n}+\alpha_i+\beta) }{(n_i-1)! \prod_{j=1,j\neq i}^p(\alpha_j-\alpha_i)_{n_j} \Gamma(\alpha_i+1) \Gamma(\sz{n}+\beta)} \frac{(-n_i+1)_k (\sz{n}+\alpha_i+\beta)_k }{(\alpha_i+1)_k }\\\times 
  \prod_{j=1,j\neq i}^p\frac{(\alpha_i-\alpha_j-n_j+1)_k}{ (\alpha_i-\alpha_j+1)_k}\frac{1}{k!} ,
 \end{multline*}
 which gives \eqref{JPTypeI}.
 
\end{proof}

\subsubsection{Multiple Orthogonal Polynomials of Type II}

\begin{pro} \label{JPII_MT}
 The Mellin transform of \(P_{\vec{n}}(x) (1-x)^\beta\) is given by
\begin{align*}
\int_0^1 P_{\vec{n}}(x) (1-x)^\beta x^{s-1} \operatorname d x = \frac{(-1)^{\sz{n}} \Gamma(\sz{n}+\beta+1)}{\prod_{i=1}^p (\alpha_i+\beta+\sz{n}+1)_{n_i}} \frac{\Gamma(s)}{\Gamma(s+\beta+\sz{n}+1)} \prod_{i=1}^p (\alpha_i+1-s)_{n_i}.
\end{align*}
\end{pro}
\begin{proof}
Performing the Mellin transform on the LHS with \(P_{\vec{n}}(x) = \sum_{k=0}^{\sz{n}} P_{\vec{n}}[k] x^k\), from \eqref{eq:Mellin_beta} we obtain
\begin{align*}
	\sum_{k=0}^{\sz{n}} P_{\vec{n}}[k] \int_0^1 x^{s+k-1} (1-x)^\beta \operatorname d x = \sum_{k=0}^{\sz{n}} P_{\vec{n}}[k] \frac{\Gamma(s+k)\Gamma(\beta+1)}{\Gamma(s+k+\beta+1)}.
\end{align*}
Taking the factor \( \Gamma(s)\Gamma(\beta+1)/\Gamma(s+\sz{n}+\beta+1) \) outside of the sum yields
\begin{align} \label{JPII_MT_COEFF}
	\frac{\Gamma(s)\Gamma(\beta+1)}{\Gamma(s+\sz{n}+\beta+1)} \sum_{k=0}^{\sz{n}} P_{\vec{n}}[k] (s)_k (s+k+\beta+1)_{\sz{n}-k} = \frac{\Gamma(s)}{\Gamma(s+\sz{n}+\beta+1)} p_{\sz{n}}(s),
\end{align}
where \(p_{\sz{n}}(s)\) is a polynomial of degree at most \( \sz{n} \). The orthogonality conditions \eqref{ortogonalidadTipoIIContinua} imply that \(p_{\sz{n}}(s)\) must vanish at the \( \sz{n} \) points \(k+\alpha_i\) for \( k \in \{1,\ldots,n_i\} \) and \( i\in\{1,\ldots,p\} \). Consequently, \( p_{\sz{n}}(s) \) can be expressed as \( p_{\sz{n}}(s) = \mathcal{P}_{\vec{n}} \prod_{i=1}^p (\alpha_i+1-s)_{n_i} \), where \(\mathcal{P}_{\vec{n}}\) is a constant.

To determine \(\mathcal{P}_{\vec{n}}\), we evaluate \eqref{JPII_MT_COEFF} at \( s=-\sz{n}-\beta \). Given that \(P_{\vec{n}}[\sz{n}] = 1\), we obtain
\begin{align*}
	\Gamma(\beta+1) (-\sz{n}-\beta)_{\sz{n}} = \mathcal{P}_{\vec{n}} \prod_{i=1}^p (\alpha_i+\beta+\sz{n}+1)_{n_i},
\end{align*}
which leads to
%\begin{align*}
\(\displaystyle
\mathcal{P}_{\vec{n}} = (-1)^{\sz{n}} \frac{\Gamma(\sz{n}+\beta+1)}{\prod_{i=1}^p (\alpha_i+\beta+\sz{n}+1)_{n_i}} \).
%\end{align*}
\end{proof}

Applying the inverse Mellin transform yields the following integral representation for the weighted type~II polynomials. Notice the analogous structure of the integral representation for the type I linear form in Theorem \ref{JI_IR}.

\begin{teo} \label{JPII_IMT}
For \(x\in (0,1)\), the function \( P_{\vec{n}}(x) (1-x)^\beta\) is given by
\begin{align}\label{eq:JPII}
 P_{\vec{n}}(x) (1-x)^\beta = \frac{(-1)^{\sz{n}} \Gamma(\sz{n}+\beta+1)}{\prod_{i=1}^p (\alpha_i+\beta+\sz{n}+1)_{n_i}} \int_{\mathcal{C}} \frac{\Gamma(s)}{\Gamma(s+\sz{n}+\beta+1)} \prod_{i=1}^p (\alpha_i+1-s)_{n_i} x^{-s} \frac{ \operatorname d s }{2 \pi \ii },
 \end{align}
where \( \mathcal{C}\) is an anticlockwise contour enclosing \( (-\infty,0]\) once. 
\end{teo}

\begin{coro}
Formula \eqref{JPTypeIIWeighted} can be derived by computing the contour integral \eqref{eq:JPII}.
\end{coro}
\begin{proof}
 The integrand only has simple poles at \(s=-k\) for \( k\in \mathbb{N}_0\). The residue at \(t=-k\) is given by
 \begin{align*}
  \frac{(-1)^k}{k!} \frac{\prod_{i=1}^p (\alpha_i+1+k)_{n_i}}{\Gamma(\sz{n}+\beta+1-k)} x^k = \frac{(-\sz{n}-\beta)_k   }{\Gamma(\sz{n}+\beta+1)} 
 \prod_{i=1}^p 
  \frac{(\alpha_i+1)_{n_i}  (n_i+\alpha_i+1)_k}{(\alpha_i+1)_k}   \frac{x^k}{k!}.
 \end{align*}
 To establish \eqref{JPTypeIIWeighted}, it then remains to note that we may replace the contour integral \(\int_{\mathcal{C}}\) by \(\lim_{M\to\infty} \int_{\mathcal{C}_M}\) in terms of anticlockwise oriented rectangles \(\mathcal{C}_M\) that connect the vertices \(-M\pm \operatorname i/2\) and \(1/2\pm \operatorname i/2\). We may do this because Stirling's formula \eqref{eq:Stirling} implies that  such that  \[\left|\frac{\Gamma(-M+ \operatorname i s) }{\Gamma(-M+ \operatorname i s+\sz{n}+\beta+1)}\right| \sim  \frac{1}{M^{\sz{n}+\beta+1}}\] as \( M\to\infty \) and thus
 \begin{align*}
  \lim_{M\to\infty} \int_{-M-\frac{\operatorname i}{2}}^{-M+\frac{\operatorname i}{2}} \frac{\Gamma(s)}{\Gamma(s+\sz{n}+\beta+1)} \prod_{i=1}^p (\alpha_i+1-s)_{n_i} x^{-s} \frac{ \operatorname d s }{2 \pi \ii } = 0,
 \end{align*}
 whenever \(x\in (0,1)\).
\end{proof}

\subsection{Hahn Multiple Orthogonal Polynomials}

Here, we rely on \S \ref{S:Hahn_0} and \S \ref{S:Hahn}.

\subsubsection{Multiple Orthogonal Polynomials of Type I}
Recall that the type I multiple Hahn polynomials are vectors of polynomials \((Q_{\vec{n}}^{(1)}, \ldots, Q_{\vec{n}}^{(p)})\), where \(\deg Q_{\vec{n}}^{(i)} \leq n_i - 1\). The linear form
\begin{align*}
Q_{\vec{n}}^{(I)}(x) = \sum_{i=1}^p Q_{\vec{n}}^{(i)}(x) \frac{\Gamma(x+\alpha_i+1)}{\Gamma(\alpha_i+1)} ,
\end{align*}
 satisfies the conditions:
\begin{align*}
\int_0^\infty Q_{\vec{n}}^{(I)}(x) \frac{\Gamma(\beta+N-x+1)}{\Gamma(x-k+1)\Gamma(N-x+1)}\d\nu_N(x) =
\begin{cases}
	0, & \text{if } k \in \{0, \ldots, |\vec{n}| - 2\}, \\
	\Gamma(\beta+1), & \text{if } k = |\vec{n}| - 1.
\end{cases}
\end{align*}
Recall that \( \nu_N=\sum_{k=0}^N \delta_k\) has mass points on a lattice \( \mathcal{L}_N = \{0,\ldots,N\}\).
Observe that we consider orthogonality with respect to \( (x-k+1)_k = \Gamma(x+1)/\Gamma(x-k+1)\) instead of the standard \(x^k\). This is more suitable for our approach.

To proceed further, we need the following analogue of Theorem \ref{LI_IR}.

\begin{pro} \label{HI_PC}
	Let us consider:
	\begin{enumerate}[\rm i)]
		\item a %polynomial combination 
		linear form
		\( F_{\vec{n}}(x) = \sum_{i=1}^p A_{\vec{n},i}(x) 
\Gamma(x+\alpha_i+1)
		\) with \( \deg A_{\vec{n},i} \leqslant n_i-1\),
		\item 
		an anticlockwise contour \( \Sigma\subset\C\) enclosing once \(\mathcal I=\cup_{i=1}^p[\alpha_i,n_i+\alpha_i-1]\),
		\item and a function \( \varphi \) analytic inside and on \( \Sigma\) that doesn't vanish on \( \mathcal I\). 
	\end{enumerate}Then there exists a unique polynomial \( p_{\sz{n}-1}(t)\) of degree at most \( \sz{n}-1\) such that
 \begin{align*}
 F_{\vec{n}}(x) = \int_{\Sigma} \frac{p_{\sz{n}-1}(t)}{\prod_{i=1}^p (\alpha_i-t)_{n_i}} \varphi(t) \Gamma(x+t+1) \frac{ \operatorname d t }{2\pi \ii}.
 \end{align*}
\end{pro}

\begin{proof}
Suppose that \( p_{\sz{n}-1}(t) \) is a polynomial of degree at most \( \sz{n}-1 \). By applying the residue theorem, we derive that
\begin{align*}
\int_{\Sigma} \frac{p_{\sz{n}-1}(t)}{\prod_{i=1}^p (\alpha_i-t)_{n_i}} \varphi(t) \Gamma(x+t+1) \frac{ \operatorname d t }{2\pi \ii} = \sum_{i=1}^p \sum_{k=0}^{n_i-1} \frac{(-1)^k p_{\sz{n}-1}(k+\alpha_i) \varphi(k+\alpha_i) }{(n_i-1-k)! k! \prod_{j=1,j\neq i}^p (\alpha_j-\alpha_i-k)_{n_j} }\Gamma(x+k+\alpha_i+1).
\end{align*}
Thus, expanding \( A_{\vec{n},i}(x) = \sum_{k=0}^{n_i-1} A_{\vec{n},i}[k] (x+\alpha_i+1)_k \) in the basis \( \{(x+\alpha_i+1)_k\}_{k=0}^{n_i-1} \), we find that 
\[ A_{\vec{n},i}[k] = \frac{(-1)^k p_{\sz{n}-1}(k+\alpha_i) \varphi(k+\alpha_i) }{(n_i-1-k)! k! \prod_{j=1,j\neq i}^p (\alpha_j-\alpha_i-k)_{n_j} }.
%p_{\sz{n}-1}(k+\alpha_i) \varphi(k+\alpha_i).
\]
 These \( \sz{n} \) interpolation conditions uniquely determine the degree \( \sz{n}-1 \) polynomial~\( p_{\sz{n}-1}(t) \).
\end{proof}

The orthogonality conditions, together with the normalization condition, allow us to explicitly determine the polynomials we have to use in order to recover the type I linear forms \[ Q_{\vec{n}}^{(I)}(x) = \sum_{i=1}^p Q_{\vec{n}}^{(i)}(x) \frac{\Gamma(x+\alpha_i+1)}{\Gamma(\alpha_i+1)}.\]

\begin{teo} \label{HI_IR}
Let \( \Sigma \subset \mathbb{C} \) be as in Proposition \ref{HI_PC} and take \( \operatorname{Re}(t) > -1 \). Then the linear form
\begin{align*}
 Q_{\vec{n}}^{(I)}(x) \coloneq  \sum_{i=1}^p Q_{\vec{n}}^{(i)}(x) \frac{\Gamma(x+\alpha_i+1)}{\Gamma(\alpha_i+1) } ,
 \end{align*}
 is given by
\begin{align}\label{eq:HI}
 Q_{\vec{n}}^{(I)}(x) = (-1)^{\sz{n}-1} (N-\sz{n}+1)! \frac{\prod_{i=1}^p (\alpha_i+\beta+\sz{n})_{n_i}}{(\beta+1)_{\sz{n}-1}} \int_\Sigma \frac{\Gamma(t+\beta+\sz{n})}{\Gamma(t+1) \Gamma(t+\beta+N+2)} \frac{\Gamma(x+t+1)}{\prod_{i=1}^p (\alpha_i-t)_{n_i}} \frac{ \operatorname d t }{2 \pi \ii }. 
\end{align} 
\end{teo}
\begin{proof}
Using Proposition \ref{HI_PC}, we can express 
\begin{align*}
Q_{\vec{n}}^{(I)}(x) = \int_{\Sigma}\frac{p_{\sz{n}-1}(t)}{\prod_{i=1}^p (\alpha_i-t)_{n_i}} \varphi(t) \Gamma(x+t+1) \frac{ \operatorname d t }{2 \pi \ii },
\end{align*}
where \( p_{\sz{n}-1}(t) \) is a polynomial of degree at most \( \sz{n}-1 \) and \( \varphi(t) \) (and the contour \(\Sigma\)) will be chosen appropriately later. The orthogonality conditions can be written as
\begin{align*}
\int_0^\infty Q_{\vec{n}}^{(I)}(x) \frac{\Gamma(\beta+N-x+1)}{\Gamma(x-s+1)\Gamma(N-x+1)} \operatorname d \nu_N(x) = 0, \quad s \in \{0, \ldots, \sz{n}-2\}.
\end{align*}
Bringing the integral inside the contour integral, we get
\begin{align*}
\int_{\Sigma}\frac{p_{\sz{n}-1}(t)}{\prod_{i=1}^p (\alpha_i-t)_{n_i}} \varphi(t) \left( \sum_{k=s}^N \frac{\Gamma(k+t+1)}{(k-s)!} \frac{\Gamma(N-k+\beta+1)}{(N-k)!} \right) \frac{ \operatorname d t }{2 \pi \ii } = 0.
\end{align*}
The inner sum can be simplified using the Chu--Vandermonde identity \eqref{Chu-Van}. Specifically,
\begin{align*}
\begin{aligned}
	\sum_{k=s}^N \frac{\Gamma(k+t+1)}{(k-s)!} \frac{\Gamma(N-k+\beta+1)}{(N-k)!} 
	&= \sum_{k=0}^{N-s} \frac{\Gamma(k+s+t+1)}{k!} \frac{\Gamma(N-s-k+\beta+1)}{(N-s-k)!} \\
	&= \Gamma(s+t+1) \Gamma(\beta+1) \frac{(s+t+\beta+2)_{N-s}}{(N-s)!}.
\end{aligned}
\end{align*}
Substituting this back into the contour integral, we have
\begin{align*}
\frac{\Gamma(\beta+1)}{(N-s)!} \int_{\Sigma}\frac{p_{\sz{n}-1}(t)}{\prod_{i=1}^p (\alpha_i-t)_{n_i}} \varphi(t) \Gamma(t+\beta+N+2) \frac{\Gamma(s+t+1)}{\Gamma(s+t+\beta+2)} \frac{ \operatorname d t }{2 \pi \ii } = 0.
\end{align*}
It is convenient to choose \( \varphi(t) = \Gamma(t+\beta+\sz{n}) / (\Gamma(t+\beta+N+2)\Gamma(t+1)) \), transforming the above to
\begin{align*}
\frac{\Gamma(\beta+1)}{(N-s)!} \int_{\Sigma}\frac{p_{\sz{n}-1}(t)}{\prod_{i=1}^p (\alpha_i-t)_{n_i}} (t+1)_s (s+t+\beta+2)_{\sz{n}-2-s} \frac{ \operatorname d t }{2 \pi \ii } = 0,
\end{align*}
which matches the contour integral derived in the proof of Theorem \ref{JI_IR}. We can then conclude from Lemma \ref{IR_lem} that \( p_{\sz{n}-1}(t) \) must be a constant, say \( p_{\sz{n}-1}(t) = p_{\sz{n}-1}[0] \). The precise value of this constant is determined by the normalization condition:
\begin{align*}
\frac{\Gamma(\beta+1)}{(N-\sz{n}+1)!} \int_{\Sigma}\frac{p_{\sz{n}-1}[0]}{\prod_{i=1}^p (\alpha_i-t)_{n_i}} \frac{(t+1)_{\sz{n}-1}}{t+\sz{n}+\beta} \frac{ \operatorname d t }{2 \pi \ii } =  \Gamma(\beta+1).
\end{align*}
Since the integrand decays as
\(\operatorname{O}(|t|^{-2})\) as \( |t|\to\infty \), we can evaluate the integral using the residue at \( t=-\sz{n}-\beta \) (note that $-\sz{n}-\beta<-1$ for \( \sz{n}\geq 2 \)):
\begin{align*}
\int_\Sigma \frac{p_{\sz{n}-1}[0]}{\prod_{i=1}^p (\alpha_i-t)_{n_i}} \frac{(t+1)_{\sz{n}-1}}{t+\sz{n}+\beta} \frac{ \operatorname d t }{2 \pi \ii } = \frac{p_{\sz{n}-1}[0]}{\prod_{i=1}^p (\alpha_i+\beta+\sz{n})_{n_i}} (-\sz{n}-\beta+1)_{\sz{n}-1}.
\end{align*}
We can therefore conclude that 
%\begin{align*}
\(
\displaystyle
p_{\sz{n}-1}[0] = (-1)^{\sz{n}-1} (N-\sz{n}+1)! \frac{\prod_{i=1}^p (\alpha_i+\beta+\sz{n})_{n_i}}{(\beta+1)_{\sz{n}-1}} \).
%\end{align*}
\end{proof}

As a byproduct we find the following constructive proof  of the previously given hypergeometric representation of type I  Hahn multiple orthogonal polynomials.
\begin{coro}
The explicit hypergeometric formula \eqref{HahnTypeI} 
follows from the evaluation of the contour integral \eqref{eq:HI}.
\end{coro}
\begin{proof}
 The integrand only has simple poles inside the contour \( \Sigma\) at \(t=\alpha_i+k\) for \(k\in\{0,\ldots,n_i-1\}\) and \(i\in\{1,\ldots,p\}\). Since \(\text{Re}(t)>-1\), the poles of \[\frac{\Gamma(t+\beta+\sz{n})}{\Gamma(t+\beta+N+2)} = \frac{1}{(t+\beta+\sz{n})_{N+2-\sz{n}}}\]
 are outside of \( \Sigma\) (assuming that \( \sz{n}\geq 2 \)). The residue at \(t=\alpha_i+k\) is given by
 \begin{multline*}
   \frac{\Gamma(k+\alpha_i+\beta+\sz{n})}{\Gamma(k+\alpha_i+1) \Gamma(k+\alpha_i+\beta+N+2)} \frac{(-1)^k}{k! (n_i-k-1)!} \frac{1}{\prod_{j=1,j\neq i}^p(\alpha_j-\alpha_i-k)_{n_j}} \Gamma(x+k+\alpha_i+1) \\
   = \frac{(\sz{n}+\alpha_i+\beta)_k \Gamma(\sz{n}+\alpha_i+\beta)}{(\alpha_i+1)_k \Gamma(\alpha_i+1) (\alpha_i+\beta+N+2)_k \Gamma(\alpha_i+\beta+N+2)} \frac{(-n_i+1)_k}{k! (n_i-1)!} 
   	\\\times  \prod_{j=1,j\neq i}^p\frac{(\alpha_i-\alpha_j-n_j+1)_k}{ (\alpha_i-\alpha_j+1)_k(\alpha_j-\alpha_i)_{n_j}}\Gamma(x+k+\alpha_i+1).
 \end{multline*}
 Consequently,
 \begin{multline*}
  Q_{\vec{n}}^{(i)}(x)  = \frac{(-1)^{\sz{n}-1} (N-\sz{n}+1)! \Gamma(\sz{n}+\alpha_i+\beta) \prod_{j=1}^p (\alpha_j+\beta+\sz{n})_{n_j} }{(\beta+1)_{\sz{n}-1} (n_i-1)!  \Gamma(\alpha_i+\beta+N+2) \prod_{j=1,j\neq i}^p(\alpha_j-\alpha_i)_{n_j}  } \\
    \times \sum_{k=0}^{n_i-1}  \frac{(-n_i+1)_k (\sz{n}+\alpha_i+\beta)_k }{(\alpha_i+1)_k (\alpha_i+\beta+N+2)_k}  	\prod_{j=1,j\neq i}^p\frac{(\alpha_i-\alpha_j-n_j+1)_k}{ (\alpha_i-\alpha_j+1)_k} \frac{(x+\alpha_i+1)_k}{k!} ,
 \end{multline*}
 which leads to the result in Theorem \ref{HahnTypeITheorem}.

\end{proof}

\subsubsection{Multiple Orthogonal Polynomials of Type II}

Recall that the type II multiple Hahn polynomials are monic polynomials \(Q_{\vec{n}}(x)\) that satisfy the orthogonality conditions
\begin{align*}
	\int_0^\infty Q_{\vec{n}}(x) \frac{\Gamma(x+k+\alpha_i+1)}{\Gamma(x+1)} \frac{\Gamma(\beta+N-x+1)}{\Gamma(N-x+1)} \operatorname{d} \nu_N(x) = 0,
\end{align*}
for \(k \in \{0, \ldots, n_i - 1\}\) and \(i \in \{1, \ldots, p\}\). Note that we consider orthogonality with respect to \((x + \alpha_i + 1)_k\) instead of the standard \(x^k\), as this will be more compatible with our approach.

The analogue of Proposition \ref{JPII_MT} is stated below. Indeed, it follows from Stirling's formula \eqref{eq:Stirling} that
\[\begin{aligned}
	\frac{\Gamma(N-Nx+\beta+1)}{\Gamma(N-Nx+1)} &\sim (N(1-x))^{\beta}, &
	\frac{\Gamma(Nx+s)}{\Gamma(Nx+1)} &\sim (Nx)^{s-1},& N &\to \infty.
\end{aligned}\]

\begin{pro} \label{HII_MT}
 Let us use the notation
\begin{align*}
\tilde{Q}_{\vec{n}}(x) = Q_{\vec{n}}(x) \frac{\Gamma(N-x+\beta+1)}{\Gamma(N-x+1)}.
\end{align*}
 Then
\begin{align*}
\begin{aligned}
 \int_0^\infty \tilde{Q}_{\vec{n}}(x) \frac{\Gamma(x+s)}{\Gamma(x+1)} \operatorname d \nu_N(x)
 = \frac{(-1)^{\sz{n}} \Gamma(\sz{n}+\beta+1)}{(N-\sz{n})! \prod_{i=1}^p (\alpha_i+\beta+\sz{n}+1)_{n_i}} \frac{\Gamma(s) \Gamma(s+N+\beta+1)}{\Gamma(s+\sz{n}+\beta+1)} \prod_{i=1}^p (\alpha_i+1-s)_{\sz{n}-1}.
\end{aligned}
\end{align*}
\end{pro}
\begin{proof}
To proceed, we will expand \( Q_{\vec{n}}(x) = \sum_{k=0}^{\sz{n}} Q_{\vec{n}}[k] (-x)_k \) in terms of \( (-x)_k = (-1)^k \Gamma(x+1)/\Gamma(x-k+1) \). This allows us to express the integrand as follows:
\begin{align*}
	\sum_{k=0}^{\sz{n}} (-1)^k Q_{\vec{n}}[k] \frac{\Gamma(x+s)}{\Gamma(x-k+1)} \frac{\Gamma(N-x+\beta+1)}{\Gamma(N-x+1)}.
\end{align*}
After evaluating the integral, we obtain 
\begin{align*}
	\sum_{k=0}^{\sz{n}} (-1)^k Q_{\vec{n}}[k] \sum_{l=k}^N \frac{\Gamma(l+s)}{(l-k)!} \frac{\Gamma(N-l+\beta+1)}{(N-l)!}.
\end{align*}
We can simplify the inner sum using the Chu--Vandermonde identity \eqref{Chu-Van}, which gives:
\begin{align*}
	\Gamma(s+k)\Gamma(\beta+1) \sum_{l=0}^{N-k} \frac{(s+k)_l}{l!} \frac{(\beta+1)_{N-k-l}}{(N-k-l)!} = \Gamma(s+k)\Gamma(\beta+1) \frac{(s+k+\beta+1)_{N-k}}{(N-k)!}.
\end{align*}
Thus, the integral can be written as
\begin{align} \label{HII_MT_COEFF}
	\Gamma(\beta+1) \frac{\Gamma(s)\Gamma(s+N+\beta+1)}{\Gamma(s+\sz{n}+\beta+1)} \sum_{k=0}^{\sz{n}} \frac{(-1)^k Q_{\vec{n}}[k]}{(N-k)!} (s)_k (s+k+\beta+1)_{\sz{n}-k} = \frac{\Gamma(s)\Gamma(s+N+\beta+1)}{\Gamma(s+\sz{n}+\beta+1)} p_{\sz{n}}(s),
\end{align}
where \( p_{\sz{n}}(s) \) is a polynomial of degree at most \(\sz{n}\). 

From the orthogonality conditions, we infer that \( p_{\sz{n}}(s) \) must vanish at the \(\sz{n}\) values \(\alpha_i + k\) for \(k \in \{0, \ldots, n_i - 1\}\) and \(i\in \{1, \ldots, p\}\). Thus, we have:
\begin{align*}
	p_{\sz{n}}(s) = \mathcal{Q}_{\vec{n}} \prod_{i=1}^p (\alpha_i + 1 - s)_{n_i},
\end{align*}
for some constant \(\mathcal{Q}_{\vec{n}} \in \mathbb{R}\).

We can determine the constant \(\mathcal{Q}_{\vec{n}}\) by setting \(s = -\sz{n} - \beta\) in \eqref{HII_MT_COEFF}. Since $Q_{\vec{n}}(x)$ is monic, and thus \(Q_{\vec{n}}[\sz{n}] = (-1)^{\sz{n}}\), it follows that:
\begin{align*}
	\Gamma(\beta+1) \frac{(-\sz{n}-\beta)_{\sz{n}}}{(N-\sz{n})!} = \mathcal{Q}_{\vec{n}} \prod_{i=1}^p (\alpha_i + \beta + \sz{n} + 1)_{n_i}.
\end{align*}
Therefore,
%\begin{align*}
\(\displaystyle
	\mathcal{Q}_{\vec{n}} = \frac{ (-1)^{\sz{n}} \Gamma(\sz{n} + \beta + 1)}{(N-\sz{n})! \prod_{i=1}^p (\alpha_i + \beta + \sz{n} + 1)_{n_i}}\).
%\end{align*}
\end{proof}

\begin{rem}
	If we make use of formula \eqref{JPTypeII}, we can now provide a constructive proof for the expression of the type II Hahn polynomials in \eqref{HahnTypeII} through the following result.
\end{rem}

\begin{pro}
Let us write \( Q_{\vec{n}}(x) = \sum_{k=0}^{\sz{n}} Q_{\vec{n}}[k] (-x)_k \) and \( P_{\vec{n}}(x) = \sum_{k=0}^{\sz{n}} P_{\vec{n}}[k] x^k \). Then the relationship between the coefficients \( Q_{\vec{n}}[k] \) and \( P_{\vec{n}}[k] \) can be given by:
\begin{align*}
Q_{\vec{n}}[k] = (-1)^{k} \frac{(N-k)!}{(N-\sz{n})!} P_{\vec{n}}[k].
\end{align*}
\end{pro}
\begin{proof}
It follows from the proofs of Propositions \ref{JPII_MT} and \ref{HII_MT}, specifically from equations \eqref{JPII_MT_COEFF} and \eqref{HII_MT_COEFF}, that we have the following identities:
\begin{align*}
	\sum_{k=0}^{\sz{n}} \frac{(-1)^k Q_{\vec{n}}[k]}{(N-k)!} (s)_k (s+k+\beta+1)_{\sz{n}-k} 
	&= \frac{(-1)^{\sz{n}} (\beta+1)_{\sz{n}}}{(N-\sz{n})! \prod_{i=1}^p (\alpha_i+\beta+\sz{n}+1)_{n_i}} \prod_{i=1}^p (\alpha_i+1-s)_{n_i} \\
	&= \frac{1}{(N-\sz{n})!} \sum_{k=0}^{\sz{n}} P_{\vec{n}}[k] (s)_k (s+k+\beta+1)_{\sz{n}-k}.
\end{align*}
We therefore have the stated result.
\end{proof} 

An application of the inversion theorem for the Mellin transform in the discrete setting, as stated in Theorem \ref{DMT_INV}, leads to the following result. Note the duality with the integral representation for the type I linear form as shown in Theorem \ref{HI_IR}.

\begin{teo} \label{HII_IR}
 For \( x \in\mathcal{L}_N\), the function
\begin{align*}
\tilde{Q}_{\vec{n}}(x) = Q_{\vec{n}}(x) \frac{\Gamma(N-x+\beta+1)}{\Gamma(N-x+1)}
\end{align*}
is given by
\begin{align*}
 \tilde{Q}_{\vec{n}}(x) = \frac{(-1)^{\sz{n}} \Gamma(\sz{n}+\beta+1)}{(N-\sz{n})! \prod_{i=1}^p (\alpha_i+\beta+\sz{n}+1)_{n_i}} \int_{\mathcal{C}} \frac{\Gamma(s) \Gamma(s+N+\beta+1)}{\Gamma(s+\sz{n}+\beta+1)} \prod_{i=1}^p (\alpha_i+1-s)_{n_i} \frac{\Gamma(x+1)}{\Gamma(x+s+1)} \frac{ \operatorname d s }{2 \pi \ii },
 \end{align*}
 where \( \mathcal{C}\) is an anticlockwise contour enclosing \( [-N,0]\) once. 
\end{teo}

This result, along with the residue theorem, leads to the conclusion below.

\begin{teo}
\label{teo:hahn_2}
 For \( x \in\mathcal{L}_N\), the function
	\(\tilde{Q}_{\vec{n}}(x) %= Q_{\vec{n}}^{(II)}(x) \Gamma(N-x+\beta+1)/\Gamma(N-x+1)
	\) 
	is given by
\begin{align}\label{eq:integral_representation_hahn_2}
\tilde{Q}_{\vec{n}}(x) =
 (-1)^{\sz{n}} \frac{\Gamma(N+\beta+1)}{(N-\sz{n})!} \frac{\prod_{i=1}^p (\alpha_i+1)_{n_i}}{\prod_{i=1}^p (\alpha_i+\beta+\sz{n}+1)_{n_i}}\pFq{p+2}{p+1}{-x,-\beta-|\vec n|,\vec \alpha+\vec n+\vec e_p}{-\beta-N,\vec \alpha+\vec e_p}{1}.
\end{align}
\end{teo}
\begin{proof}
 It follows from Theorem \ref{HII_IR} and the residue theorem that
\begin{align*}
\begin{aligned}
 \tilde{Q}_{\vec{n}}(x) &= \frac{(-1)^{\sz{n}} \Gamma(\sz{n}+\beta+1)}{(N-\sz{n})! \prod_{i=1}^p (\alpha_i+\beta+\sz{n}+1)_{n_i}} \sum_{l=0}^N \frac{(-1)^l}{l!} \frac{\Gamma(N+\beta+1-l)}{\Gamma(\sz{n}+\beta+1-l)} \prod_{i=1}^p (\alpha_i+1+l)_{n_i} (x-l+1)_l 
 ,
 % \\
 %& = \frac{\Gamma(N+\beta+1)}{(N-\sz{n})! \prod_{i=1}^p (\alpha_i+\beta+\sz{n}+1)_{n_i}} \sum_{l=0}^N \frac{(-x)_l}{l!} \frac{(-\sz{n}-\beta)_l}{(-N-\beta)_l} \frac{(\vec{n}+\vec{\alpha}+1)_l}{(\vec{\alpha}+1)_l} \prod_{i=1}^p (\alpha_i+1)_{n_i} .
\end{aligned}
\end{align*}
which we may write as
\begin{align}
    \tilde{Q}_{\vec{n}}(x) = (-1)^{\sz{n}} \frac{\Gamma(N+\beta+1)}{(N-\sz{n})!} \frac{\prod_{i=1}^p (\alpha_i+1)_{n_i}}{\prod_{i=1}^p (\alpha_i+\beta+\sz{n}+1)_{n_i}} \sum_{l=0}^N  \prod_{i=1}^p 
\frac{ (n_i+\alpha_i+1)_l}{(\alpha_i+1)_l}  \frac{(-\sz{n}-\beta)_l}{(-N-\beta)_l} \frac{(-x)_l}{l!}.
\end{align}
This coincides with the expression in \eqref{eq:integral_representation_hahn_2}, because both series get truncated after $l=x\in\mathcal{L}_N$ due to the Pochhammer symbol $(-x)_l$. 
\end{proof}

\section*{Conclusions and Outlook}
The paper focuses on two main objectives. Firstly, it establishes explicit hypergeometric expressions for the type I Hahn multiple orthogonal polynomials, as demonstrated in Theorem \ref{HahnTypeITheorem}. This result is pivotal for understanding the properties and applications of these polynomials.
Secondly, 
%utilizing 
applying
the residue theorem and the Mellin transform, the paper derives contour integral representations for various families of muliple orthogonal polynomials within the classical families. Specifically, it provides integral representations for the Laguerre of the first kind, Jacobi--Piñeiro, and both types I and II Hahn multiple orthogonal polynomials. These integral formulas constitute significant contributions to the field, leading to explicit hypergeometric representations that enhance the understanding and computational aspects of these polynomial families.

We are currently working on providing explicit hypergeometric expressions and integral representations for the discrete descendants of the Hahn polynomials in the Askey scheme. These include the multiple Meixner polynomials of the first and second kinds, as well as the Kravchuk and Charlier polynomials.

An unexplored frontier lies in mixed multiple orthogonal polynomials, which extend the classical families of multiple orthogonal polynomials discussed in this paper to the mixed multiple case. Currently, explicit hypergeometric expressions for mixed Hahn multiple orthogonal polynomials remain elusive, as do complex contour integral representations for these polynomials. This area presents an open avenue for future research, promising new insights into the interplay between different orthogonal polynomial families and their applications.

\section*{Acknowledgments} AB acknowledges Centre for Mathematics of the University of Coimbra funded
	by the Portuguese Government through FCT/MCTES, DOI: 10.54499/UIDB/00324/2020.

AF and JEFD acknowledge Center for Research and Development
	in Mathematics and Applications (CIDMA) from University of Aveiro funded by the Portuguese Foundation
	for Science and Technology (FCT) through projects DOI: 10.54499/UIDB/04106/2020 and DOI:
	10.54499/UIDP/04106/2020. Additionally, JEFD acknowledges PhD contract
	DOI: 10.54499/UI/BD/152576/2022 from FCT.

JEFD and MM acknowledge research project [PID2021- 122154NB-I00], \emph{Ortogonalidad y Aproximación con Aplicaciones en Machine Learning y Teoría de la Probabilidad} funded by \href{https://doi.org/10.13039/501100011033}{MICIU/AEI/10.13039/501100011033} and by "ERDF A Way of making Europe”

TW acknowledges PhD project 3E210613 funded by BOF KU Leuven.

\end{document}